\documentclass{article}

\usepackage{arxiv}

\usepackage[utf8]{inputenc} 
\usepackage[T1]{fontenc}    
\usepackage{hyperref}       
\usepackage{url}            
\usepackage{booktabs}       
\usepackage{amsfonts}       
\usepackage{nicefrac}       
\usepackage{microtype}      
\usepackage{lipsum}		
\usepackage{graphicx}
\usepackage[numbers]{natbib}
\usepackage{doi}
\usepackage{xcolor}
\usepackage{amsmath}
\usepackage{amsthm}
\usepackage{dsfont}
\usepackage{algorithm}
\usepackage{algpseudocode}
\usepackage{setspace}
\usepackage{graphicx}
\usepackage{subcaption}

\DeclareFontFamily{U}{mathx}{\hyphenchar\font45}
\DeclareFontShape{U}{mathx}{m}{n}{
      <5> <5.5> <6> <7> <7.5> <8> <9> <10> <10.5>
      <10.95> <12> <14.4> <17.28> <20.74> <24.88>
      mathx10
      }{}
\DeclareSymbolFont{mathx}{U}{mathx}{m}{n}
\DeclareFontSubstitution{U}{mathx}{m}{n}
\DeclareMathAccent{\widebar}{0}{mathx}{"73}

\usepackage{dsfont}
\usepackage{relsize}
\newtheorem{Theorem}{Theorem}
\newtheorem{Lemma}{Lemma}
\newtheorem{Corollary}{Corollary}
\newtheorem{Example}{Example}

\usepackage{booktabs}
\usepackage{multirow}

\usepackage{algpseudocode,algorithm,algorithmicx}

\usepackage[labelformat=simple]{subcaption}

\DeclareCaptionLabelFormat{subcaptionlabel}{\normalfont(\textbf{#2}\normalfont)}
\usepackage{float}

\title{Probabilistic closed-form formulas for pricing nonlinear payoff variance and volatility derivatives under Schwartz model with time-varying log-return volatility}

\author{ \href{https://orcid.org/0000-0002-9186-6891}{\includegraphics[scale=0.06]{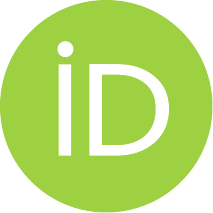}\hspace{1mm}Nontawat Bunchak} \\
	Department of Statistics \\ Faculty of Science, Kasetsart University \\ Bangkok 10900, Thailand \\
	\texttt{nontawat.bunch@ku.th} \\
	\And
	\href{https://orcid.org/0000-0001-7797-4418}{\includegraphics[scale=0.06]{orcid.pdf}\hspace{1mm}Udomsak Rakwongwan} \\
	Department of Mathematics \\ Faculty of Science, Kasetsart University \\ Bangkok 10900, Thailand\\
	\texttt{udomsak.ra@ku.th} \\
	\And
	\href{https://orcid.org/0000-0002-0986-2518}{\includegraphics[scale=0.06]{orcid.pdf}\hspace{1mm}Phiraphat Sutthimat} \thanks{Corresponding author} \\
	Department of Mathematics \\
	Faculty of Science, Kasetsart University \\
	 Bangkok 10900, Thailand \\
	\texttt{phiraphat.sut@ku.th} \\ 
}

\renewcommand{\shorttitle}{Pricing nonlinear payoff variance and volatility derivatives under Schwartz model}

\hypersetup{
pdftitle={Pricing nonlinear payoff variance and volatility derivatives under Schwartz model},
pdfsubject={Conditional moments},
pdfauthor={Nontawat Bunchak, Udomsak Rakwongwan, Phiraphat Sutthimat},
pdfkeywords={Schwartz one-factor model, Time-varying volatility, Noncentral chi-square, Volatility swaps, Volatility options, Vega},
}

\begin{document}
\maketitle

\begin{abstract}
This paper presents closed-form analytical formulas for pricing volatility and variance derivatives with nonlinear payoffs under discrete-time observations. The analysis is based on a probabilistic approach assuming that the underlying asset price follows the Schwartz one-factor model, where the volatility of log-returns is time-varying. A difficult challenge in this pricing problem is to solve an analytical formula under the assumption of time-varying log-return volatility, resulting in the realized variance being distributed according to a linear combination of independent noncentral chi-square random variables with weighted parameters. By utilizing the probability density function, we analytically compute the expectation of the square root of the realized variance and derive pricing formulas for volatility swaps. Additionally, we derive analytical pricing formulas for volatility call options. For the payoff function without the square root, we also derive corresponding formulas for variance swaps and variance call options. Additionally, we study the case of constant log-return volatility; simplified pricing formulas are derived and sensitivity with respect to volatility (vega) is analytically studied. Furthermore,we propose simple closed-form approximations for pricing volatility swaps under the Schwartz one-factor model. The accuracy and efficiency of the proposed methods are demonstrated through Monte Carlo simulations, and the impact of price volatility and the number of trading days on fair strike prices of volatility and variance swaps is investigated across various numerical experiments.
\end{abstract}



\section{Introduction} \label{sec_intro}
\medskip

Commodity price volatility has become a focal point for academics, investors, and economists due to its critical role in volatility trading and risk management. Unlike stocks, bonds, and other traditional financial instruments, commodities exhibit exceptionally large price swings around a long-term equilibrium~\cite{blanco2001mean}, driving demand for volatility-based instruments such as variance and volatility swaps. These forward contracts allow traders to take pure positions on volatility through a simple payoff: at maturity, the long pays a fixed variance strike and receives any realized variance above that level, while the short pays the realized variance and receives the fixed strike. Investor interest in these swaps surged in 1998 following the collapse of Long-Term Capital Management (LTCM) and the ensuing spike in market turbulence; initially dominated by hedge funds selling volatility to dealers~\cite{mehta1999equity}, the market has since grown into a vibrant arena with extensive research on pricing and hedging these contracts.  

Research on variance swap pricing has overwhelmingly focused on equity markets, where a series of foundational studies has established robust valuation and replication methods. Demeterfi et al.~\cite{demeterfi1999guide} demonstrated that, under continuous price paths, variance swaps can be statically replicated using portfolios of vanilla options and that volatility swaps follow by dynamically trading those positions, also deriving an analytical fair-value formula that accounts for realistic volatility skews. Windcliff, Forsyth and Vetzal~\cite{windcliff2006pricing} applied numerical partial integro-differential equation methods to price volatility products and explored delta and delta-gamma hedging strategies for volatility swaps. Javaheri, Wilmott and Haug~\cite{javaheri2004garch} valued and hedged volatility swaps within a GARCH(1,1) stochastic volatility framework by determining the first two moments of realized variance through a PDE approach and approximating expected realized volatility via Brockhaus and Long\rq s second-order convexity  djustment~\cite{brockhaus2000volatility}. Th\'{e}or\'{e}t, Zabr\'{e} and Rostan~\cite{theoret2002pricing} then provided an analytical solution for pricing volatility swaps under the same GARCH model and applied it to the S\&P60 Canada index. Subsequent work by Carr and Madan~\cite{jouini2001option} confirmed exact replication of variance swaps through vanilla options, Ma and Xu~\cite{ma2010efficient} incorporated control variate techniques in stochastic volatility settings, Zhu and Lian~\cite{zhu2011closed} derived closed-form solutions under Heston\rq s two-factor model, Rujivan and Zhu~\cite{rujivan2012simplified} offered a simplified analytical approach for discretely sampled variance swaps, Zheng and Kwok~\cite{zheng2014closed} extended pricing to generalized swaps with simultaneous jumps in price and variance, Shen and Siu~\cite{shen2013pricing} included stochastic interest rates and regime switches, and Chan and Platen~\cite{chan2015pricing} developed explicit formulas under the modified constant elasticity of variance model. Despite this extensive equity-focused literature, commodity markets have received little attention; the sole exception is Swishchuk~\cite{swishchuk2010variance}, who used the Brockhaus--Long approximation to derive variance and volatility swap formulas for energy assets following a continuous-time GARCH(1,1) process.

In the context of volatility and variance derivatives for commodities, models for stochastic commodity price dynamics differ from those for other asset classes because they explicitly account for the convenience yield and allow for multiple sources of uncertainty. Schwartz~\cite{schwartz1997stochastic} introduced three such models. In the first model the log of the spot price follows a mean reverting process around a constant convenience yield. The second model treats the convenience yield itself as stochastic. The third model further extends uncertainty by allowing interest rates to evolve randomly over time. In 2016, Chunhawiksit and Rujivan~\cite{chunhawiksit2016pricing} proposed an analytical closed‐form solution for pricing discretely‐sampled variance swaps on commodities under the one‐factor Schwartz model, defining realized variance via squared percentage returns, validating its financial meaningfulness, and demonstrating substantial computational efficiency gains over Monte Carlo (MC) methods. In the same year, Weraprasertsakun and Rujivan~\cite{weraprasertsakun2017closed} extended this method by deriving a closed‐form formula for pricing discretely‐sampled variance swaps on commodities, now defining realized variance in terms of squared log returns, showing that their solution produces financially meaningful fair delivery prices throughout the model\rq s parameter space and dramatically reduces the computational burden compared to MC simulations, thereby offering market practitioners a highly efficient and implementable analytical tool.

However, the results proposed in~\cite{chunhawiksit2016pricing,weraprasertsakun2017closed} and also~\cite{chumpong2021analytical,duangpan2022analytical,sutthimat2022closed} are limited to variance swaps because the closed-form formulas were derived by solving the Feynman--Kac formula to obtain conditional moments of the one-factor Schwartz model, whereas volatility swaps cannot be obtained through conditional moments and instead require half moments of log return. And, of course, the closed-form formula for the half moments of the log returns under the one-factor Schwartz model is not currently available, so a new technique must be introduced to address this issue.

In this paper, we derive closed-form analytical formulas for pricing discretely sampled volatility and variance derivatives under the one-factor Schwartz model with time-varying log-return volatility by expressing the realized variance as a linear combination of independent noncentral chi-square random variables and using their probability density function (PDF) to compute the expectation of its square root. As a result, we obtain explicit pricing formulas for volatility swaps, volatility call options, variance swaps and variance call options (with simplified expressions and analytical vega in the constant-volatility case), propose simple closed-form approximations for volatility swaps, and demonstrate the accuracy and efficiency of our methods through extensive MC simulations.

The remainder of this paper is organized as follows. Section~\ref{sec_Schwartz} reviews the one‐factor Schwartz model. Section~\ref{sec_VarAndVol} introduces variance and volatility swaps. In Section~\ref{sec_pdf_rv}, we derive the probability density function of the realized variance and present its conditional moments. Section~\ref{sec_pricing_swaps} uses these results to obtain closed‐form pricing formulas for volatility and variance swaps and their corresponding options, and provides sensitivity analyzes. Section~\ref{sec_application} validates our analytical findings with numerical experiments and discussion. Finally, Section~\ref{sec_conclusion} summarizes our contributions and suggests avenues for future research.


\section{Schwartz model} \label{sec_Schwartz}
\medskip

In this section we recall Schwartz one-factor model~\cite{schwartz1997stochastic} for commodity price dynamics, which we adopt under the risk-neutral probability space $(\Omega,\mathcal{F},(\mathcal{F}_t)_{t\in[0,T]},\mathbb{Q})$ to ensure absence of arbitrage in variance swap pricing. Under $\mathbb{Q}$, the commodity spot price $S_t$ follows the stochastic differential equation (SDE)
\begin{equation} \label{sde_st}
  \mathrm{d}S_t = \kappa(\mu - \ln S_t)\,S_t\,\mathrm{d}t + \sigma\,S_t\,\mathrm{d}W_t,\quad S_0>0,\;t\in(0,T],
\end{equation}
where $\kappa>0$ is the speed of mean reversion to the long-run log-price level $\mu$, $\sigma>0$ is the volatility, and $(W_t)_{t\in[0,T]}$ is a standard Brownian motion on $(\Omega,\mathcal{F},(\mathcal{F}_t)_{t\in[0,T]},\mathbb{Q})$. Defining the log‐price process $X_t=\ln S_t$ and applying It\^{o}\rq s lemma to \eqref{sde_st} yields the Ornstein--Uhlenbeck (OU) process
\begin{equation}\label{sde_xt}
    \mathrm{d}X_t
    = \kappa\bigl(\alpha - X_t\bigr)\,\mathrm{d}t
    + \sigma\,\mathrm{d}W_t,
    \quad X_0 = \ln S_0,
\end{equation}
where $\kappa \ln S_t$ is the instantaneous convenience yield at time $t$ and
\begin{equation*}\label{alpha}
    \alpha \;=\; \mu \;-\;\frac{\sigma^2}{2\,\kappa}\,.
\end{equation*}
It is important to note that, unlike the original OU process sense, a commodity does not behave like a conventional asset, and its spot price (or equivalently the logarithm of the spot price) serves as the state variable on which contingent claims are written. 


\section{Variance and volatility swaps} \label{sec_VarAndVol}
\medskip

Following the collapse of LTCM in late 1998, when implied stock-index volatility reached unprecedented levels, the market for variance and volatility swaps began to expand. Unlike traditional stock options, these swaps provide pure exposure to future volatility~\cite{demeterfi1999guide} and have attracted investment banks and other financial institutions. Investors use them to speculate on future volatility, to trade the spread between realized and implied volatility or to hedge volatility exposure in other positions. Today, variance and volatility swaps are actively quoted across a broad range of assets such as stock indices, currencies and commodities. In the remainder of this chapter, we provide a comprehensive overview of their definitions, valuation strategies and practical implementations.

In a mathematical context, we define discretely sampled volatility and variance based on the log-returns of the underlying asset price. The primary objective of this paper is to estimate the log-return realized variance, commonly referred to as the realized variance, which is defined as
\begin{equation} \label{def_rv}
    RV \equiv RV_d(t_1,N,T) := \frac{1}{T} \sum_{i=2}^N \ln{\!\left(\frac{S_{t_i}}{S_{t_{i-1}}}\right)} \times 100^2 =\frac{AF}{N-1} \sum_{i=2}^N \ln {\!\left(\frac{S_{t_i}}{S_{t_{i-1}}}\right)} \times 100^2,
\end{equation}
for $t_1 \in (0,T)$, where $S_{t_i}$ is the underlying asset\rq s closing price at time $t_i$ for the $i$-th observation belonging to the total number of observations $N \geq 2$. The term $AF=\frac{N-1}{T}=\frac{1}{\Delta t}$ is the annualization factor, used to standardize the realized variance over the time horizon $T$.


\subsection{Variance swaps}
\medskip

Under a risk-neutral martingale measure $\mathbb{Q}$, and assuming that $r(t_1)$ is the time-varying risk-free interest rate at time $t_1$, the value of a variance swap at time $t_1$ can be expressed as the expected present value of its future payoff:
\begin{equation*} \label{payoff_var}
    V_{t_1,\mathrm{var}} = e^{-\int_{t_1}^{T} r(s)\, \mathrm{d}s} \, \mathbb{E}_{t_1}^{\mathbb{Q}} \big[ RV_d - \widebar{K}_{\mathrm{var}} \big] \times L_{\mathrm{var}},
\end{equation*}
where $L_{\mathrm{var}}$ denotes the notional amount of the swap, measured in dollars per annualized variance point. More specifically, $L_{\mathrm{var}}$ represents the amount received by the holder at maturity for each unit by which the realized variance $RV_d$ exceeds the strike $\widebar{K}_{\mathrm{var}}$.

Because there is no upfront cost to enter into a variance swap, we have $V_{t_1,\mathrm{var}} = 0$. Therefore, the fair strike price of a variance swap is given by
\begin{equation} \label{def_kvar}
    \widebar{K}_{\mathrm{var}} := \mathbb{E}_{t_1}^{\mathbb{Q}}\big[RV_d\big],
\end{equation}
where $\mathbb{E}_{t_1}^{\mathbb{Q}}[X]$ denotes the conditional expectation of a random variable $X$ with respect to the filtration $\mathcal{F}_{t_1}$ under the risk-neutral martingale measure $\mathbb{Q}$. Thus, the valuation problem for a variance swap reduces to computing the expected value of the future realized variance in the risk-neutral world.


\subsection{Volatility swaps}
\medskip

Similar to the valuation of a variance swap, the value of a volatility swap at time $t_1$ can be expressed as the expected present value of the future payoff,
\begin{equation*} \label{payoff_vol}
    V_{t_1,\mathrm{vol}} = e^{-\int_{t_1}^T r(s)\, \mathrm{d}s} \, \mathbb{E}_{t_1}^{\mathbb{Q}}\big[\sqrt{RV_d} - \widebar{K}_{\mathrm{vol}}\big] \times L_{\mathrm{vol}},
\end{equation*}
where $L_{\mathrm{vol}}$ is the notional amount of the swap, denominated in dollars per annualized volatility point. Specifically, $L_{\mathrm{vol}}$ represents the amount that the holder of the contract receives at maturity if $\sqrt{RV_d}$ exceeds the strike $\widebar{K}_{\mathrm{vol}}$ by one unit.

Since there is no upfront cost to enter into a volatility swap, we set $V_{t_1,\mathrm{vol}} = 0$. Consequently, the fair strike price of a volatility swap is given by
\begin{equation} \label{def_kvol}
    \widebar{K}_{\mathrm{vol}} := \mathbb{E}_{t_1}^{\mathbb{Q}}\big[\sqrt{RV_d}\big].
\end{equation}


\subsection{Practical applications}
\medskip

Volatility exhibits several attractive features for trading. First, it rises with risk and uncertainty and tends to increase more after bad news than after good news. Second, it follows a mean reverting process so high volatility levels decrease and low levels increase. Finally, there is a negative correlation between volatility and asset prices so volatility remains high after large downward moves in the market~\cite{demeterfi1999guide}. Variance and volatility swaps allow investors to profit from or hedge against changes in future volatility. For example, a dealer who writes an option can face losses if volatility spikes and the option must be repurchased at a higher price. To transfer this risk, the dealer can enter a volatility swap with a hedge fund. The hedge fund takes the opposite position because it expects volatility to fall and seeks to earn a profit. As the most direct instruments for trading volatility, variance and volatility swaps are central to modern financial markets and will continue to play a key role.


\section{The PDF of the realized variance} \label{sec_pdf_rv}
\medskip

Consider $X_{t_{i}}$ follows the OU stochastic process~\eqref{sde_xt} with mean $\mathbb{E}\big[X_{t_{i}}\big]$ and variance $\mathbb{VAR}\big[X_{t_{i}}\big]$. Moreover, a covariance of $X_{t_{i}}$ and $X_{t_{i-1}}$ is denoted by $\mathbb{COV}\big[X_{t_i},X_{t_{i-1}}\big]$. Let $Z_i := X_{t_{i}}-X_{t_{i-1}} = \ln{S_{t_{i}}}-\ln{S_{t_{i-1}}}$ also known as the log-return. We obtain 
\begin{equation} \label{z}
    \widebar{Z}_i = \widebar{\mu}_i+\widebar{\sigma}_i^{2} \, \big(W_{t_{i}}-W_{t_{i-1}}\big) \sim \mathcal{N}\big( \, \widebar{\mu}_i ,\widebar{\sigma}_i^{2}\big),
\end{equation}
where 
\begin{equation} \label{mean_z}
    \widebar{\mu}_i := \mathbb{E}\big[ X_{t_{i}}\big]-\mathbb{E}\big[X_{t_{i-1}}\big],
\end{equation}
and
\begin{equation} \label{variance_z}
    \widebar{\sigma}_i^{2} :=\mathbb{VAR}\big[X_{t_{i}}\big]+\mathbb{VAR}\big[X_{t_{i-1}}\big]-2 \, \mathbb{COV}\big[X_{t_i},X_{t_{i-1}}\big],
\end{equation}
at time $t_i \in (0,T]$ for $i=2,3,\ldots,N$. So this means that $\widebar{Z}_i$ is a normally distributed random variable with mean $\widebar{\mu}_i$ and variance $\widebar{\sigma}_i^{2} $ for all $i=2,3,\ldots,N$.  

By the definition of the noncentral chi-square random variables, we have
\begin{equation} \label{yi}
    \widebar{Y}_i := \left(\frac{\widebar{Z}_i}{\widebar{\sigma}_i}\right)^2 \sim \mathcal{NC}_{\chi^2_1}(\widebar{\delta}_i),
\end{equation}
for all $i=2,3,\ldots,N$ where $\widebar{Y_i}$ is a noncentral chi-square distribution with degrees of freedom $\nu_i$ and noncentrality parameter $\widebar{\delta}_i$ where $\nu_i=1$ then
\begin{equation} \label{degree_of_freedom}
    \nu := \sum^N_{i=2} \nu_i = N-1,
\end{equation}
and
\begin{equation} \label{noncentrality}
    \widebar{\delta}_i := \left( \frac{\widebar{\mu}_i}{\widebar{\sigma}_i} \right)^2,
\end{equation}
for $i=2, 3, \ldots, N$. The log-return realized variance defined in~\eqref{def_rv} can be shown in terms of a linear combination of noncentral chi-square random variables, by applying~\eqref{z}--\eqref{noncentrality}, as
\begin{equation} \label{rv}
    RV_d(t_1,N,T) = \sum_{i=2}^{N} \widebar{\alpha}_i \widebar{Y}_i,
\end{equation}
where the weighted parameter can be defined by
\begin{equation} \label{alpha_i}
    \widebar{\alpha}_i := \frac{100^2}{T}\widebar{\sigma}^2_i,
\end{equation}
for $i = 2, 3, \ldots, N$.


\subsection{Laguerre expansions for the PDF of \texorpdfstring{$RV_d$}{}}
\medskip

Over the past several decades, the distribution of linear combinations of noncentral chi-square random variables, or equivalently, quadratic forms in normal random variable vectors, has been the subject of extensive research. The PDFs associated with the distribution have been expressed in various representations, accompanied by many proposed methods for their efficient computation. For instance, in 1961, Shah and Khatri~\cite{shah1961distribution} introduced the distribution for the definite case using power series expansions. In 1962, Ruben~\cite{ruben1962probability} proposed a representation for the infinite case in terms of chi-squared series. Subsequently, in 1963, Shahe~\cite{shah1963distribution} extended their approach to the infinite case and compared their results with those of Ruben. Then, in 1967, Kotz et al.~\cite{kotz1967series} represented the distribution in terms of various expressions including power series expansions, Laguerre series expansions, chi-squared series, and noncentral chi-squared series. Furthermore, in 1977, Davis~\cite{davis1977differential} presented percentage point approximation tables for the distribution under a differential equation approach. To represent the PDF of a linear combination of independent noncentral chi-square random variables with positive weights, as given in~\eqref{rv}, one of the most efficient computational methods is developed from the work of Kotz et al. (1967) and Davis (1977), and is represented in terms of Laguerre series expansions proposed by Casta\~{n}o-Mart\'{i}nez and L\'{o}pez-Bl\'{a}zquez in 2005~\cite{castano2005distribution}.  

For our study, we apply the presented approach by Casta\~{n}o-Mart\'{i}nez and L\'{o}pez-Bl\'{a}zquez~\cite{castano2005distribution} to obtain the PDF of $RV_d$, as shown in the following theorem. The generalized Laguerre function is given by the following series representation 
\begin{equation*}
    \mathbf{L}^{(a)}_k (x) = \sum^\infty_{m=0} (-1)^m\frac{(a+k)!}{(a+m)!\, (k-m)!\, m!} \, x^m,
\end{equation*}
where the parameter $a>-1$ and the fractional order $k \in (n-1,n)$ for some $n \in \mathbb{N}$.
\begin{Theorem} \label{thm_pdf_rv}
    Let $y>0$. The PDF of $RV_d$ written as $\widebar{f}_\nu$, satisfying
    \begin{equation*} \label{prob}
        \mathbb{Q}(RV_d \leq y) = \int^y_0 \widebar{f}_\nu (\zeta) \, \mathrm{d} \zeta,
    \end{equation*}
    and can be illustrated as
    \begin{equation} \label{pdf_rv}
        \widebar{f}_\nu (y) \equiv \widebar{f}_\nu^{(\widebar{\beta},\widebar{\mu}_0)} (y) := \frac{e^{-\frac{y}{2\widebar{\beta}}} y^{\frac{\nu}{2}-1}}{(2\widebar{\beta})^{\frac{\nu}{2}}} \sum^\infty_{k=0} \frac{\Gamma(k-1)}{\Gamma \! \left(\frac{\nu}{2}+k\right)} \widebar{c}_k \, \mathbf{L}^{\left(\frac{\nu}{2}-1\right)}_k \! \! \left( \frac{\nu y}{4 \widebar{\beta} \widebar{\mu}_0} \right),
    \end{equation}
    with three parameters $\nu=N-1$, $\widebar{\beta}>0$, and $\widebar{\mu}_0>0$, where $\Gamma (x) =\int_0^{\infty} t^{x-1} e^{-t} \, \mathrm{d} t$ is the gamma function. The coefficient $\widebar{c}_k \equiv \widebar{c}_k \big(\nu,\widebar{\sigma}_i,\widebar{\mu}_i;\widebar{\beta},\widebar{\mu}_0\big)$, for $k = 1, 2, 3, \ldots$, depends on time-varying log-return volatility $\widebar{\sigma}_i$ and mean $\widebar{\mu}_i$ for $i=2, 3, \ldots, N$, satisfying the following recurrent relation,
    \begin{align}
        &\widebar{c}_k = \frac{1}{k} \sum_{j=0}^{k-1} \widebar{c}_j \widebar{d}_{k-j}, \quad \text{ for } k\geq1, \label{recurrentck} \\
        &\widebar{c}_0 = \left( \frac{\nu}{2\widebar{\mu}_0} \right)^{\frac{\nu}{2}} \prod^N_{i=2} \left( 1+\frac{\widebar{\alpha}_i}{\widebar{\beta}}\left(\frac{\nu}{2\widebar{\mu}_0}-1\right)\right)^{-\frac{1}{2}} \exp{ \left( -\frac{1}{2} \sum^N_{i=2} \frac{ \widebar{\delta}_i \widebar{\alpha}_i (\nu-2\widebar{\mu}_0)}{2\widebar{\beta} \widebar{\mu}_0+\widebar{\alpha}_i (\nu-2\widebar{\mu}_0)} \right)}, \label{recurrentc0} \\
        &\widebar{d}_j = \sum^N_{i=2} \frac{\nu_i}{2} \left( \frac{\widebar{\beta}-\widebar{\alpha}_i}{\widebar{\beta}-\widebar{\alpha}_i\big(\frac{\nu}{2\widebar{\mu}_0}-1\big)} \right)^{j}
        -\frac{j\widebar{\beta}\nu}{4\widebar{\mu}_0} \sum^N_{i=2} \widebar{\delta}_i \widebar{\alpha}_i (\widebar{\beta}_i-\widebar{\alpha}_i)^{j-1} \left( \frac{2\widebar{\mu}_0}{2\widebar{\beta}\widebar{\mu}_0+\widebar{\alpha}_i(\nu-2\widebar{\mu}_0)}\right)^{j+1}, \quad \text{ for } j\geq1. \label{recurrentdj}
    \end{align}
\end{Theorem}
\begin{proof}
    We will derive $\widebar{f}_\nu$, i.e. the PDF of $RV_d$, by applying the presented approach by Casta\~{n}o-Mart\'{i}nez and L\'{o}pez-Bl\'{a}zquez~\cite{castano2005distribution} in Section~3 as follows.
    
    Firstly, we consider a solution of the mean-reverting OU process~\eqref{sde_xt},
    \begin{equation*} \label{solu_x}
        X_{t_i} = X_0e^{-\kappa t_i}+\kappa \alpha e^{-\kappa t_i} \int_0^{t_i} e^{\kappa s} \, \mathrm{d} s+\sigma e^{-\kappa t_i} \int_0^{t_i} e^{\kappa s} \, \mathrm{d} W_t.
    \end{equation*}
    We know that $X_{t_i}$ is normally distributed with mean
    \begin{equation} \label{mean_x}
        \mathbb{E}[X_{t_{i}}] = e^{-\kappa t_i}x_0+(1-e^{-\kappa t_{i}}) \alpha
    \end{equation}
    for given $X_{t_0}=x_0$, and variance
    \begin{equation} \label{variance_x}
        \mathbb{VAR}[X_{t_{i}}]  = \frac{\sigma^2}{2\kappa}(1-e^{-2\kappa t_{i}})
    \end{equation}
    at time $t_i \in (0,T]$ for $i=2, 3, \ldots,N$. In addition, since $X_{t_i}$ and $X_{t_{i-1}}$ are jointly distributed normal random variables, we then obtain that its covariance can be expressed as
    \begin{equation} \label{cov_x}
        \mathbb{COV}[X_{t_i},X_{t_{i-1}}] = \mathbb{VAR}[X_{t_{i-1}}] e^{-\kappa \Delta t},
    \end{equation}
    where $\Delta t = t_{i}-t_{i-1}$ (see Franco~\cite{FrancoMaximumLE}).
    According to the second term on RHS of the log-return in~\eqref{z}, the increments $\big(W_{t_{j}}-W_{t_{j-1}}\big)\sim\mathcal{N}\big(0,\Delta t \big)$ and $\big(W_{t_{k}}-W_{t_{k-1}}\big)\sim\mathcal{N}\big(0,\Delta t \big)$ for all $j,k= 2, 3, \ldots, N$ and $j \neq k$, are independent by using the property of a standard Brownian motion. Therefore, $\widebar{Z}_i $ for $ i=2, 3, \ldots, N$, is the sequence of independent and identically distributed (i.i.d) random variables. Thus, a result of a random variable $RV_d$ as illustrated in~\eqref{def_rv} is a linear combination of independent noncentral chi-square random variables with weighted parameters $\widebar{\alpha}_i>0$ as defined in~\eqref{rv}, for $i = 2, 3, \ldots, N$.
    
    Secondly, in Section~3 of~\cite{castano2005distribution}, the constants $n, \alpha_i, \delta_i, p$ and the parameters, $\beta_i$, and $\mu_0$, are replaced by the following: $n=N, \alpha_i=\widebar{\alpha}_i, \delta_i=\widebar{\delta}, p=\nu/2, \beta=\widebar{\beta}_i, \mu_0=\widebar{\mu}_0$. Consequently, by applying equations (3.2), (3.3), (3.4)a, and (3.4)b in~\cite{castano2005distribution}, the PDF of $RV_d$ can be constructed as shown in~\eqref{pdf_rv} where the recurrent coefficient $\widebar{c}_k$ for $k=0, 1, 2, \ldots$, can be calculated by using~\eqref{recurrentck} and~\eqref{recurrentc0} as well as the coefficient $\widebar{d}_j$ for $k=1, 2, 3, \ldots$, can be calculated by using~\eqref{recurrentdj}.
\end{proof}


\subsection{The explicit formula for \texorpdfstring{$\mathbb{E} [RV^\ell_d]$}{}}
\medskip

By utilizing the Laguerre expansion~\eqref{pdf_rv} along with some properties of the generalized hypergeometric function
\begin{equation} \label{hyper_fun}
    _p\mathbf{F}_q(a_1,a_2,\ldots,a_p;b_1,b_2,\ldots,b_q;z) = \sum_{k=0}^\infty \frac{(a_1)_k(a_2)_k\cdots(a_p)_k}{(b_1)_k(b_2)_k\cdots(b_q)_k} \frac{z^k}{k!},
\end{equation}
where $(\cdot)_k$ denotes the usual Pochhammer symbol~\cite{slater1966generalized}, the explicit formula for the $\ell$-th conditional moment of $RV_d$ is derived in the following theorem.
\begin{Theorem} \label{thm_moment}
    For any $\ell \in \mathbb{R}_+$, the raw moment of $RV_d$ can be illustrated as
    \begin{equation} \label{generalmoment}
        \mathbb{E}^\mathbb{Q}_{t_1} [RV^\ell_d] = (2\widebar{\beta})^\ell \sum^\infty_{k=0} \left(\frac{\nu}{2\widebar{\mu}_0}\right)^k \frac{\Gamma \! \left(\ell+k+\frac{\nu}{2}\right)}{\Gamma \! \left(k+\frac{\nu}{2}\right)}{_2\mathbf{F}_1} \! \! \left(-k,1-k-\frac{\nu}{2};1-k-\frac{\nu}{2}-\ell;\frac{2\widebar{\mu}_0}{\nu}\right)\widebar{c}_k,   
    \end{equation}
    where the coefficients $\widebar{c}_k$, $k=0, 1, 2,\ldots$, are chosen according to~\eqref{recurrentck}--\eqref{recurrentdj} and the parameters $\nu>0$ and $\widebar{\beta}>0$ can be arbitrarily chosen.
\end{Theorem}
\begin{proof}
    We define
    \begin{equation*} \label{F}
    \widebar{F}_{k,\ell}:=\frac{1}{2\widebar{\beta}} \frac{\Gamma(k-1)}{\Gamma \! \left(\frac{\nu}{2}+k\right)} \widebar{c}_k \, \mathbf{L}^{(\frac{\nu}{2}-1)}_k \! \! \left( \frac{\nu y}{4 \widebar{\beta} \widebar{\mu}_0} \right) \ e^{-\frac{y}{2\widebar{\beta}}} \, y^{\frac{\nu}{2}-1+\ell},
    \end{equation*}
    for $k = 0, 1, 2, \ldots, \ell \in \mathbb{R}_+$ where $\widebar{c}_k$\rq s, $\widebar{\beta}$, $\nu$ are given in Theorem~\ref{thm_pdf_rv}. Employing the symbolic computing package in MATHEMATICA yields
    \begin{equation} \label{intF}
        \int^\infty_0 \widebar{F}_{k,\ell} (y) \, \mathrm{d} y = (2\widebar{\beta})^\ell \left(\frac{\nu}{2\widebar{\mu}_0}\right)^k \frac{\Gamma \! \left(\ell+k+\frac{\nu}{2}\right)}{\Gamma \! \left(k+\frac{\nu}{2}\right)}{_2\mathbf{F}_1} \! \! \left(-k,1-k-\frac{\nu}{2};1-k-\frac{\nu}{2}-\ell;\frac{2\widebar{\mu}_0}{\nu}\right)\widebar{c}_k .   
    \end{equation}
    According to the uniformly convergent series $\widebar{f}_n$ derived in Theorem~\ref{thm_pdf_rv}, we obtain that the series $\sum^\infty_{k=0}\widebar{F}_{k,\ell} (y)$ converges uniformly to $y^\ell\widebar{f}_\nu^{(\widebar{\beta},\widebar{\mu}_0)} (y)$ i.e.,
    \begin{equation} \label{Fconf}
        y^\ell\widebar{f}_\nu^{(\widebar{\beta},\widebar{\mu}_0)} (y) = \sum^\infty_{k=0}\widebar{F}_{k,\ell} (y).    
    \end{equation}
    By using the property of uniformly convergent series and applying~\eqref{intF},~\eqref{Fconf} yields
    \begin{align*}
        \mathbb{E}^\mathbb{Q}_{t_1} \big[RV^\ell_d\big] &= \int^\infty_0 y^\ell\widebar{f}_\nu^{(\widebar{\beta},\widebar{\mu}_0)} (y) \, \mathrm{d} y = \sum^\infty_{k=0} \int^\infty_0 \widebar{F}_{k,r} (y) \, \mathrm{d} y \\
        &= (2\widebar{\beta})^\ell \sum^\infty_{k=0} \left(\frac{\nu}{2\widebar{\mu}_0}\right)^k \frac{\Gamma \! \left(\ell+k+\frac{\nu}{2}\right)}{\Gamma \! \left(k+\frac{\nu}{2}\right)}{_2\mathbf{F}_1} \! \! \left(-k,1-k-\frac{\nu}{2};1-k-\frac{\nu}{2}-\ell;\frac{2\widebar{\mu}_0}{\nu}\right)\widebar{c}_k.
    \end{align*}
    This completes the proof.
\end{proof}


\subsection{Approximates for truncation errors of \texorpdfstring{$\mathbb{E} \big[RV^\ell_d\big]$}{}}
\medskip

On a computation of $\mathbb{E} \big[RV^\ell_d\big]$ by machine, we have to estimate the error from substituting an infinite sum with a finite sum which is a method known as truncation error. An approximation for the truncation errors of $\mathbb{E} \big[RV^\ell_d\big]$ will be derived in this subsection by applying the approaches stated in Subsection 3.1 of Casta\~{n}o-Mart\'{i}nez and L\'{o}pez-Bl\'{a}zquez (see~\cite{castano2005distribution}).

we define
\begin{equation*} \label{error}
    \varepsilon_{k_1,k_2}^{(\widebar{\beta}, \widebar{\mu}_0)} (\ell,\nu) = (2\widebar{\beta})^\ell \sum^{k_2}_{k=k_1+1} \left(\frac{\nu}{2\widebar{\mu}_0}\right)^k \frac{\Gamma \! \left(\ell+k+\frac{\nu}{2}\right)}{\Gamma \! \left(k+\frac{\nu}{2}\right)}{_2\mathbf{F}_1} \! \! \left(-k,1-k-\frac{\nu}{2};1-k-\frac{\nu}{2}-\ell;\frac{2\widebar{\mu}_0}{\nu}\right)\widebar{c}_k,
\end{equation*}
for any $k_i = 0, 1, 2,\ldots$, $i=1,2$, such that $k_1+1=k_2$. Thus, a truncation error of order $k$ of $\mathbb{E} \big[RV^\ell_d\big]$ is denoted by $\varepsilon_{k_1,k_2}^{(\widebar{\beta}, \widebar{\mu}_0)} (\ell,\nu)$. To approximate our truncation error, bounds of $\widebar{c}_k$ for $k=0, 1, 2,\ldots$, given in~\eqref{recurrentck}--\eqref{recurrentdj}, can be derived below.
\begin{Lemma} \label{lemma_abs_ck}
    The coefficient $\widebar{c}_k$, $k=0,1,2,\ldots$, satisfies
    \begin{equation} \label{abs_ck}
        |\widebar{c}_k| \leq \zeta^k \left( \frac{2k+\nu}{2k} \right)^k \left( \frac{2k+\nu}{\nu} \right)^{\frac{\nu}{2}} b_0\big(\nu, \widebar{\beta}, \widebar{\mu}_0, \widebar{\delta}, \zeta\big),
    \end{equation}
    for $k=0,1,2,\ldots$, where $\zeta = \max_{i \in \{2,\ldots,N\}} \left| \frac{1-\frac{\widebar{\alpha}_i}{\widebar{\beta}}}{1+\frac{\widebar{\alpha}_i}{\widebar{\beta}}\left( \frac{\nu}{2\widebar{\mu}_0}-1 \right)} \right|$ and
    \begin{equation} \label{b0}
        \widebar{b}_0(\nu, \widebar{\beta}, \widebar{\mu}_0, \widebar{\delta}, \zeta) = \left( \frac{\nu}{2\widebar{\mu}_0} \right)^{\frac{\nu}{2}} \prod^N_{i=2} \left| 1+\frac{\widebar{\alpha}_i}{\widebar{\beta}}\left(\frac{\nu}{2\widebar{\mu}_0}-1\right) \right|^{-\frac{1}{2}}   \exp{ \left( \frac{\widebar{\mu}_0 \widebar{\delta}}{\nu \zeta}-\frac{1}{4} \sum^N_{i=2} \frac{ \widebar{\delta}_i \widebar{\alpha}_i(\nu-2\widebar{\mu}_0)}{\widebar{\beta}+\widebar{\alpha}_i (\nu-2\widebar{\mu}_0)} \right) },
    \end{equation}
    where $N\geq2$ is a positive integer and \ $\widebar{\delta}=\sum^N_{i=2} \widebar{\delta}_i$. In addition, 
    \begin{equation*}
        0<\zeta<1, \text{ if } \widebar{\mu}_0 \geq \frac{\nu}{4} \text{ and } \widebar{\beta}> \frac{1}{2}\left( 2-\frac{\nu}{2 \widebar{\mu}_0} \right) \max_{i \in \{2,\ldots,N\}} \widebar{\alpha}_i.
    \end{equation*}
\end{Lemma}
\begin{proof}
    We will derive bounds of coefficient $\widebar{c}_k$, $k=0,1,2,\ldots$, by applying the presented approach in Lemma 3.1 written by Casta\~{n}o-Mart\'{i}nez and L\'{o}pez-Bl\'{a}zquez (see~\cite{castano2005distribution}) as follows. The constants $n, \alpha_i, \delta_i, p$ and parameter $\beta_i$, and $\mu_0$, are placed by the following: $n=N, \alpha_i=\widebar{\alpha}_i, \delta_i=\widebar{\delta}_i, p=\nu/2, \beta=\widebar{\beta}_i, \mu_0=\widebar{\mu}_0$. Then, by applying Lemma 3.1, bounds of coefficient $\widebar{c}_k$, $k=0,1,2,\ldots$, can be constructed as shown in~\eqref{abs_ck} where the mention at the end can be derived by using Remark 3.1 given by Casta\~{n}o-Mart\'{i}nez and L\'{o}pez-Bl\'{a}zquez~\cite{castano2005distribution}.
\end{proof}

According to Lemma~\ref{lemma_abs_ck}, we further define
\begin{equation} \label{bound}
    B^{(\widebar{\beta}, \widebar{\mu}_0)}_{k_1,k_2} (\ell, \nu;\zeta) := (2\widebar{\beta})^\ell \, \widebar{b}_0\big(\nu, \widebar{\beta}, \widebar{\mu}_0, \widebar{\delta}, \zeta\big) \sum^{k_2}_{i=k_1+1} \widebar{p}_k \big(\ell, \nu, \widebar{\mu}_0, \zeta\big),
\end{equation}
for all $k_i=0,1,2,\ldots$, $i=1,2$, such that $k_1+1=k_2$ and $\zeta>0$, where
\begin{equation} \label{pk}
    \widebar{p}_k\big(\ell, \nu, \widebar{\mu}_0, \zeta\big) = \left(\frac{\nu}{2\widebar{\mu}_0}\right)^k \frac{\Gamma \! \left(\ell+k+\frac{\nu}{2}\right)}{\Gamma \! \left(k+\frac{\nu}{2}\right)}\left| {_2\mathbf{F}_1} \! \! \left(-k,1-k-\frac{\nu}{2};1-k-\frac{\nu}{2}-\ell;\frac{2\widebar{\mu}_0}{\nu}\right) \right| \left( \frac{2k+\nu}{2k} \right)^k \left( \frac{2k+\nu}{\nu} \right)^{\frac{\nu}{2}} \zeta^k.
\end{equation}
Due to Lemma~\ref{lemma_abs_ck}, our bound for truncation error can be derived as the following theorem.
\begin{Theorem} \label{thm_abs_error}
    Suppose that $\widebar{\mu}_0 \geq \frac{\nu}{4}$ and $\widebar{\beta}> \frac{1}{2}\left( 2-\frac{\nu}{2 \widebar{\mu}_0} \right) \max_{i \in \{2,\ldots,N\}} \widebar{\alpha}_i$. Then
    \begin{equation} \label{abs_error}
    \left| \varepsilon_{K,\infty}^{(\widebar{\beta}, \widebar{\mu}_0)} (\ell,\nu) \right| \leq B^{(\widebar{\beta}, \widebar{\mu}_0)}_{K,\infty} (\ell, \nu;\zeta),
    \end{equation}
    for all $K \in \mathbb{Z}_+$ and $\zeta = \max_{i \in \{2,\ldots,N\}} \left| \frac{1-\frac{\widebar{\alpha}_i}{\widebar{\beta}}}{1+\frac{\widebar{\alpha}_i}{\widebar{\beta}}\left( \frac{\nu}{2\widebar{\mu}_0}-1 \right)} \right|$.
\end{Theorem}
\begin{proof}
    To proceed~\eqref{abs_error}, we utilize~\eqref{generalmoment} and~\eqref{error}--\eqref{pk}.
\end{proof}


\section{Pricing volatility and variance swaps} \label{sec_pricing_swaps}
\medskip

By applying Theorem~\ref{thm_pdf_rv} which introduces the PDF of $RV_d$, the fair strike price of volatility and variance swaps, defined in~\eqref{def_kvol} and~\eqref{def_kvar} respectively, can be expressed as
\begin{equation} \label{int_kvar}
    \widebar{K}_{var} =\int^\infty_0 y\widebar{f}^{(\widebar{\beta}, \widebar{\mu}_0)}_n (y) \, \mathrm{d} y,
\end{equation}
and
\begin{equation} \label{int_kvol}
    \widebar{K}_{vol} =\int^\infty_0 y\widebar{f}^{(\widebar{\beta}, \widebar{\mu}_0)}_n (\sqrt{y}) \, \mathrm{d} y,
\end{equation}
respectively.

Although $\widebar{K}_{var}$  and $\widebar{K}_{vol}$ can be enumerated by using several techniques of numerical integration for both improper integrals above ,\eqref{int_kvar} and~\eqref{int_kvol}, those numerical techniques can lead to complications. To avoid those complications, we simplify the integral into a series of an infinite summation of a generalized hypergeometric function. It is obtained by analyzing the conditional expectation order $\ell$ of $RV_d$ known as a raw moment proposed in Theorem~\ref{thm_moment}. 

In this section, the fair strike price of a volatility swap and variance swap is derived by setting $\ell=1/2$ and $\ell=1$ in~\eqref{generalmoment}, respectively. Additionally, we analyze the closed-form formulas for the fair strike price of both swaps in two scenarios, depending on the variance of the log-return: one where the variance changes over time, and another where it remains constant.


\subsection{Pricing formulas for volatility and variance swaps: Time-varying log-return volatility} \label{sec_K_vary}
\medskip 

Due to the solution of the SDE in~\eqref{sde_xt}, $X_{t_i}$, for $i = 2,3, \ldots, N$, is a normally distributed random variable with time-varying log-return volatility. The difference between $X_{t_i}$ and $X_{t_{i-1}}$, denoted as $\widebar{Z}_i$ in~\eqref{z}, for $i = 2,3,\ldots,N$, also follows a normal distribution with mean given by~\eqref{mean_x} and variance given by~\eqref{variance_x}. The covariance of this difference is provided by~\eqref{cov_x}.

In this subsection, these time-varying log-return volatility is used to analyze the explicit formulas for volatility and variance swaps, where $RV_d$ is considered as a linear combination of independent noncentral chi-square random variables with weighted parameters $\widebar{\alpha}_i>0$ for $i = 2, 3,\ldots, N$. 


\subsubsection{Volatility swaps price: \texorpdfstring{$K^1_{vol}$}{}}
\medskip

By using the probabilistic approach based on the PDF~\ref{pdf_rv}, the following theorem provides a closed-form formula for pricing a volatility swap under time-varying log-return volatility. The corresponding result for the case of constant log-return volatility is presented in the subsequent corollary.

\begin{Theorem} \label{thm_Kvol1_1}
    According to Theorem~\ref{thm_moment}, suppose that $\widebar{\mu}_0=\frac{\nu}{2}$ and $\widebar{\beta}>\frac{1}{2}\max_{i \in \{2,\ldots,N\}} \widebar{\alpha}_i$, and we choose $\ell=\frac{1}{2}$. The fair strike price of a volatility swap can be shown as
    \begin{equation} \label{KvolSwap1_1}
        K^1_{vol_1} \equiv \widebar{K}^{\left(\widebar{\beta}, \frac{\nu}{2}\right)}_{vol_1} (\nu,T;\widebar{\sigma}_i)= \sqrt{2\widebar{\beta}} \, \sum^\infty_{k=0} \left(-1\right)^k \frac{\Gamma \! \left(k+\frac{\nu+1}{2}\right)}{\Gamma \! \left(k+\frac{\nu}{2}\right)}{_2\mathbf{F}_1} \! \left(-k,1-k-\frac{\nu}{2};k-\frac{\nu+1}{2};1\right)\widebar{c}_k,
    \end{equation}
    where the coefficients $\widebar{c}_k \equiv \widebar{c}_k \big(\nu,\widebar{\sigma}_i,\widebar{\mu}_i;\widebar{\beta},\frac{\nu}{2}\big), k=0,1,2,\ldots$ are calculated by using~\eqref{recurrentck}--\eqref{recurrentdj} and the parameters $\widebar{\sigma}_i,\widebar{\mu}_i,\nu$ and $\widebar{\alpha}_i$, $i=2,3,\ldots,N$, are defined in~\eqref{mean_z},~\eqref{variance_z},~\eqref{degree_of_freedom} and~\eqref{alpha_i} respectively.
\end{Theorem}
\begin{proof}
    We set $\widebar{\mu}_0=\nu/2$, $\widebar{\beta}>(1/2)\max_{i \in \{2,\ldots,N\}} \widebar{\alpha}_i$, and $\ell=1/2$ in~\eqref{generalmoment}; we then immediately obtain~\eqref{KvolSwap1_1}.
\end{proof}

\begin{Corollary} \label{col_Kvol1_2}
    According to Theorem~\ref{thm_Kvol1_1}, assume that $\widebar{\sigma}_2=\widebar{\sigma}_3=\ldots=\widebar{\sigma}_N=c$, the fair strike price of a volatility swap can be shown as
    \begin{equation} \label{KvolSwap1_2}
        K^1_{vol_2} \equiv \widebar{K}^{\left(\widebar{\beta}, \frac{\nu}{2}\right)}_{vol_2} (\nu,T;c)= \sqrt{\frac{2c}{T}} \, \frac{\Gamma \! \left(\frac{\nu+1}{2}\right)}{\Gamma \! \left(\frac{\nu}{2}\right)} \times 100,
    \end{equation}
    where $c>0$ can be arbitrarily chosen and $\nu > 0$ is defined in~\eqref{degree_of_freedom}.
\end{Corollary}
\begin{proof}
    According to the proof of Theorem~\ref{thm_Kvol1_1}, since we set $\widebar{\sigma}^2_i=c^2$, we choose $\widebar{\beta}=(100^2/T)\,c^2$ for all $i=2,3,\ldots,N$. From~\eqref{recurrentdj} and~\eqref{recurrentck}, we then obtain that $\widebar{d_j}=0$ for all $j=1,2,3\ldots$, and $\widebar{c}_k=0$ for all $k=1,2,3,\ldots$. This means that $\widebar{K}^1_{vol_2}$ exists only if $k=0$. Obviously, we have $\widebar{c_k}=1$ and ${_2\mathbf{F}_1} \! \left(-k,1-k-\nu/2;k-(\nu+1)/2;1\right)=1$ when $k=0$ for all $\nu>0$ by~\eqref{recurrentc0} and~\eqref{hyper_fun} respectively. Then we immediately obtain~\eqref{KvolSwap1_2}.
\end{proof}


\subsubsection{Variance swaps price: \texorpdfstring{$K^1_{var}$}{}}
\medskip

Closed-form formulas for pricing variance swaps under both time-varying and constant log-return volatility are presented in the following theorem and corollary.

\begin{Theorem} \label{thm_Kvar1_1}
    According to Theorem~\ref{thm_moment}, suppose that $\widebar{\mu}_0=\frac{\nu}{2}$ and $\widebar{\beta}>\frac{1}{2}\max_{i \in \{2,\ldots,N\}} \widebar{\alpha}_i$, and we choose $\ell=1$. The fair strike price of a variance swap can be shown as
    \begin{equation} \label{KvarSwap1_1}
        K^1_{var_1} \equiv \widebar{K}^{\left(\widebar{\beta}, \frac{\nu}{2}\right)}_{var_1} (\nu,T;\widebar{\sigma}_i)= 2\widebar{\beta} \sum_{k=0}^{\infty}(-1)^k\frac{\Gamma \! \left(1+k+\frac{\nu}{2}\right)}{\Gamma \! \left(k+\frac{\nu}{2}\right)}{_2\mathbf{F}_1} \! \left(-k,1-k-\frac{\nu}{2};-k-\frac{\nu}{2};1\right)\widebar{c}_k,
    \end{equation}
    where the coefficients $\widebar{c}_k \equiv \widebar{c}_k \big(\nu,\widebar{\sigma}_i,\widebar{\mu}_i;\widebar{\beta},\frac{\nu}{2}\big), k=0,1,2,\ldots$ are calculated by using~\eqref{recurrentck}--\eqref{recurrentdj} and the parameters $\widebar{\sigma}_i,\widebar{\mu}_i,\nu$ and $\widebar{\alpha}_i$, $i=2,3,\ldots,N$, are defined in~\eqref{mean_z},~\eqref{variance_z},~\eqref{degree_of_freedom} and~\eqref{alpha_i} respectively.
\end{Theorem}
\begin{proof}
    The equation~\eqref{KvarSwap1_1} is acquired by setting~\eqref{generalmoment} as follows: $\widebar{\mu}_0=\nu/2$, $\widebar{\beta}>(1/2)\max_{i \in \{2,\ldots,N\}} \widebar{\alpha}_i$, and $\ell=1$.
\end{proof}

\begin{Corollary} \label{col_Kvar1_2}
    According to Theorem~\ref{thm_Kvol1_1}, assume that $\widebar{\sigma}_2=\widebar{\sigma}_3=\ldots=\widebar{\sigma}_N=c$, the fair strike price of a variance swap can be shown as
    \begin{equation} \label{KvarSwap1_2}
       K^1_{var_2} \equiv \widebar{K}^{\left(\widebar{\beta}, \frac{\nu}{2}\right)}_{var_2} (\nu,T;c) = \frac{2c^2}{T} \frac{\Gamma \! \left(\frac{\nu}{2}+1 \right)}{\Gamma \! \left(\frac{\nu}{2} \right)} \times 100^2,
    \end{equation}
    where $c>0$ can be arbitrarily chosen and $\nu > 0$ is defined in~\eqref{degree_of_freedom}.
\end{Corollary}
\begin{proof}
    Due to the equation~\ref{KvarSwap1_1}, it is easy to obtain~\eqref{KvarSwap1_2} by setting all the same parameters and following the presented approach of the proof from the previous Corollary~\ref{col_Kvol1_2}.
\end{proof}


\subsection{Pricing formulas for volatility and variance swaps: Constant log-return volatility} \label{sec_K_constant}
\medskip

In this subsection, we explore the other explicit solutions of volatility and variance swaps by assuming that the variances of log-returns are set to be the same. 

Referring to~\eqref{z}, we assume that $\widebar{Z}_i$ is a normally distributed random variable with mean $\widebar{\mu}_i$ and given variance $\widebar{\sigma}^2_i=\widebar{\sigma}^2_N$ for $i=2,3,\ldots,N$, where $\widebar{\mu}_i$ and $\widebar{\sigma}_N$ are computed by~\eqref{mean_z} and~\eqref{variance_z} respectively. In other words, for $i=2,3,\ldots,N$, we choose the variance at $i=N$ i.e. at time $t_N$, to represent the variance of log-returns for all $i$-th observations. Since $\widebar{Z}_i$ for $i=2,3,\ldots,N$ forms a sequence of i.i.d. random variables as shown in the proof of Theorem~\ref{thm_pdf_rv}, we can define a random variable
\begin{equation} \label{sqrt_w}
    \sqrt{\widebar{W}}:=\sqrt{\sum_{i=2}^{N} \Bigg( \frac{{\widebar{Z}_i}}{\widebar{\sigma}_N} \Bigg)^2} \sim \mathcal{NC}_{\chi_{\eta}} \! \left(\sqrt{\widebar{\lambda}}\right),
\end{equation}
where $ \sqrt{\widebar{W}}$ follows the noncentral chi distribution with degrees of freedom $\eta$ and noncentrality parameter $\sqrt{\widebar{\lambda}}$.
Furthermore, a random variable defined by 
\begin{equation} \label{w}
    \widebar{W}:=\sum_{i=2}^{N} \Bigg( \frac{{\widebar{Z}_i}}{\widebar{\sigma}_N} \Bigg)^2 \sim \mathcal{NC}_{\chi^2_{\eta}}(\widebar{\lambda})
\end{equation}
is distributed according to the noncentral chi-square distribution with degree $\eta$ of freedom and noncentrality parameter $\widebar{\lambda}$
where
\begin{equation} \label{degree_of_freedom_w}
    \eta:=N-1
\end{equation}
and
\begin{equation} \label{noncentrality_w}
    \widebar{\lambda} := \sum^{N-1}_{i=2} \frac{\widebar{\mu}^2_i}{\widebar{\sigma}^2_N}.
\end{equation}

The interesting results are discussed in the following theorems when $RV_d$ is a random variable that can be written in terms of a linear combination of independent noncentral chi-square distributions without weighted parameters.


\subsubsection{Volatility swaps price: \texorpdfstring{$K^2_{vol}$}{}}
\medskip

Since $\sqrt{RV_d}$ can be represented in terms of $\sqrt{\widebar{W}}$, the following theorems provide closed-form formulas for pricing volatility swaps under constant log-return volatility, in the cases where the sum of time-varying log-return means is greater than zero and equal to zero, respectively. 

\begin{Theorem} \label{thm_Kvol2_1} 
    Suppose that $\displaystyle \sum^{N-1}_{i=2} \widebar{\mu}_i^2 > 0$ and $\widebar{\sigma}^2_i=\widebar{\sigma}^2_N$ for all $i=2,3,\ldots,N$. The fair strike price of a volatility swap can be shown as
    \begin{equation} \label{KvolSwap2_1}
       K^2_{vol_1} \equiv \widebar{K}_{vol_1} (\eta,T;\widebar{\sigma}_N) = \widebar{\sigma}_N \sqrt{\frac{\pi}{2T}} \, \mathbf{L}^{\big( \frac{\eta}{2}-1\big)}_{\frac{1}{2}} \! \! \left(-\frac{\widebar{\lambda}}{2} \right) \times 100,
    \end{equation}
    where  $\eta > 0$ and $\widebar{\lambda}>0$ are defined in~\eqref{degree_of_freedom_w} and~\eqref{noncentrality_w} respectively.
\end{Theorem}

\begin{proof}
    Due to~\eqref{sqrt_w}, we have the following explicit form (see Johnson et al.~\cite{johnson1995continuous})
    \begin{equation} \label{expected_sqrt_w}
        \mathbb{E}_{t_1}
        ^\mathbb{Q}[\sqrt{\widebar{W}}] = \sqrt{\frac{\pi}{2}} \, \mathbf{L}^{\frac{\eta}{2}-1}_{\frac{1}{2}} \! \! \left(-\frac{\widebar{\lambda}}{2} \right) >0
    \end{equation}
    where $\widebar{\lambda}$ is given in~\eqref{noncentrality_w}. We note that the conditional expectation above is always positive by the property of the explicitly of the Laguerre function (see~Mirevki and Boyadiiev~\cite{mirevski2010some}). Thus, the realized volatility can be expressed as 
    \begin{equation*} \label{sqrt_rv_constant}
        \sqrt{RV_d} = \widebar{\sigma}_N \sqrt{\frac{\widebar{W}}{T}} \times 100.
    \end{equation*}
    In consequence, we get~\eqref{KvolSwap2_1} by using~\eqref{expected_sqrt_w}.
\end{proof}

\begin{Theorem} \label{thm_Kvol2_2} 
    Suppose that $\displaystyle \sum^{N-1}_{i=2} \widebar{\mu}_i^2 = 0$ and $\widebar{\sigma}^2_i=\widebar{\sigma}^2_N$ for all $i=2,3,\ldots,N$. The fair strike price of a variance swap can be shown as
    \begin{equation} \label{KvolSwap2_2}
        K^2_{vol_2} \equiv \widebar{K}_{vol_2} (\eta,T;\widebar{\sigma}_N) = \widebar{\sigma}_N \sqrt{\frac{2}{T}} \, \frac{\Gamma \! \left(\frac{\eta+1}{2}\right)}{\Gamma \! \left(\frac{\eta}{2}\right)} \times 100,
    \end{equation}
    where  $\eta > 0$ is defined in~\eqref{degree_of_freedom_w}.
\end{Theorem}
\begin{proof}
    We proceed with the same approach as in the proof of the previous Theorem~\ref{thm_Kvol2_1}, but now we replace $\sum^{N-1}_{i=2} \widebar{\mu}_i^2 = 0$. We then obtain $\widebar{\lambda}=0$ that 
    \begin{equation} \label{L0}
        \mathbf{L}^{\big( \frac{\eta}{2}-1 \big)}_{\frac{1}{2}} (0) = \frac{2}{\sqrt{\pi}} \frac{\Gamma \! \left(\frac{\eta+1}{2}\right)}{\Gamma \! \left(\frac{\eta}{2}\right)},   
    \end{equation}
    where $\mathbf{L}^{(a)}_{b} (0) = \Gamma(1+a+b)/\big(\,\Gamma(1+a)\Gamma(1+b)\big)$ with $a=1/2$, $b=(\eta/2)-1$ and $\Gamma(3/2)=\sqrt{\pi}/2$ (see Mirevki and Boyadiiev~\cite{mirevski2010some}). Hence, we obtain~\eqref{KvolSwap2_2} by substituting~\eqref{L0} on the RHS of~\eqref{expected_sqrt_w}. Note that the proof can be completed by considering that $\sqrt{\widebar{W}}$ is distributed according to the central chi distribution with degrees of freedom $\eta$ when $\widebar{\lambda}=0$.
\end{proof}


\subsubsection{Variance swaps price: \texorpdfstring{$K^2_{var}$}{}}
\medskip

The following theorems are closed-form formulas for pricing variance swaps under constant log-return volatility, in the cases where the sum of time-varying log-return means is greater than zero and equal to zero, respectively.

\begin{Theorem} \label{thm_Kvar2_1} 
    Suppose that $\displaystyle \sum^{N-1}_{i=2} \widebar{\mu}_i > 0$ and $\widebar{\sigma}^2_i=\widebar{\sigma}^2_N$ for all $i=2,3,\ldots,N$. The fair strike price of a variance swap can be shown as
    \begin{equation} \label{KvarSwap2_1}
        K^2_{var_1} \equiv \widebar{K}_{var_1} (\eta,T;\widebar{\sigma}_N) = \frac{\widebar{\sigma}^2_N}{T} \, (\eta+\widebar{\lambda}) \times 100^2,
    \end{equation}
    where $\eta > 0$ and $\widebar{\lambda}>0$ are defined in~\eqref{degree_of_freedom_w} and~\eqref{noncentrality_w} respectively.
\end{Theorem}
\begin{proof}
    According to~\eqref{w}, the explicit form of conditional expectation of $\widebar{W}$ can be expressed as (see~Stuart et al.~\cite{stuart1999kendall})
    \begin{equation} \label{expected_w}
        \mathbb{E}_{t_1}^\mathbb{Q}[\widebar{W}] = \eta+\widebar{\lambda}.
    \end{equation}
    We can show that the log-return realized variance~\eqref{rv} is simplified by
    \begin{equation*} \label{rv_constant}
        RV_d = \frac{\widebar{\sigma}^2_N W}{T} \times 100^2.
    \end{equation*}
    We then get~\eqref{KvarSwap2_1} by applying~\eqref{expected_w}.
\end{proof}

\begin{Theorem} \label{thm_Kvar2_2} 
    Suppose that $\displaystyle \sum^{N-1}_{i=2} \widebar{\mu}_i^2 = 0$ and $\widebar{\sigma}^2_i=\widebar{\sigma}^2_N$ for all $i=2,3,\ldots,N$. The fair strike price of a variance swap can be shown as
    \begin{equation} \label{KvarSwap2_2}
        K^2_{var_2} \equiv \widebar{K}_{var_2} (\eta,T;\widebar{\sigma}_N) = \frac{\widebar{\sigma}^2_N}{T} \, \eta \times 100^2,
    \end{equation}
    where  $\eta > 0$ is defined in~\eqref{degree_of_freedom_w}.
\end{Theorem} 
\begin{proof}
    This proof is similar to the proof of Theorem~\ref{thm_Kvar2_1} above. But now, we replace $\sum^{N-1}_{i=2} \widebar{\mu}_i^2 = 0$ in~\eqref{noncentrality_w} then $\widebar{\lambda}=0$ which~\eqref{KvarSwap2_2} is obtained immediately. Note that we can proof directly by considering that $\widebar{W}$ is distributed according to the central chi-square distribution with degrees of freedom $\eta$ when $\widebar{\lambda}=0$.
\end{proof}

According to Corollary~\ref{col_Kvol1_2} and~\ref{col_Kvar1_2}, if we choose the arbitrary constant $c=\widebar{\sigma}_N$, we will get Theorem~\ref{thm_Kvol2_1} and~\ref{thm_Kvar2_1} respectively. Besides, in~\cite{rujivan2021analytically} written by Rujivan and Rakwongwan, both corollaries can be implied to Theorem~2.3 when $c=\sigma^2\Delta t$ and $\nu=N-1$, and Theorem~2.7 when $c=\widebar{\beta}$, $j=2$ and $\nu=N-1$, respectively. In their work, the parameters are obtained by using the Black-Scholes model with a time-varying risk-free interest rate.


\subsubsection{Vega of \texorpdfstring{$K^2_{vol}$}{} and \texorpdfstring{$K^2_{var}$}{}} \label{sec_vega_k}
\medskip

In financial mathematics, to measure the sensitivity of a derivative instrument's value, we commonly study its partial derivative with respect to the parameter of interest. In this third-level section, we focus on the first-order partial derivative with respect to the price volatility $\sigma$, known as vega, denoted by $\mathcal{V}_{\mathbf{V}} := \frac{\partial \mathbf{V}}{\partial \sigma}$ (see Hull~\cite{hull2016options}). In other words, vega can be interpreted as the amount of money invested in the derivative instrument per share, indicating the profit or loss resulting from a one percentage point increase or decrease in volatility. We note that an increase in volatility suggests a higher chance of the underlying asset reaching extreme price levels, leading to a corresponding rise in the derivative instrument's value. Conversely, reduced volatility tends to lower the derivative instrument's value.

In the following corollaries, exact formulas for vega of volatility and variance swaps, $\mathcal{V}_{K^2_{vol}}$ and $\mathcal{V}_{K^2_{vol}}$, are produced by using Theorems~\ref{thm_Kvol2_1}--\ref{thm_Kvar2_2}.

\begin{Corollary} \label{col_vega_Kvol2}
    According to Theorem~\ref{thm_Kvol2_1} and~\ref{thm_Kvol2_2}, we get the following two results:
    \begin{enumerate}
        \item if $\displaystyle \sum^{N-1}_{i=2} \widebar{\mu}_i^2 > 0$ then   
        \begin{equation} \label{vega_Kvol2_1}
            \mathcal{V}_{K^2_{vol_1}}= \frac{1}{\sigma} \left( \widebar{\lambda} {\widebar{\sigma}_N} \sqrt{\frac{\pi}{2 T}} \, \mathbf{L}^{\frac{\eta}{2}}_{-\frac{1}{2}} \! \! \left(- \frac{\widebar{\lambda}}{2} \right) \times 100 + \widebar{K}^2_{vol_1} \right),
        \end{equation}
        where $\widebar{\lambda}$ is defined in~\eqref{noncentrality_w}.
        \item if $\displaystyle \sum^{N-1}_{i=2} \widebar{\mu}_i^2 = 0$ then        
        \begin{equation} \label{vega_Kvol2_2}
            \mathcal{V}_{\widebar{K}^2_{vol_2}}= \frac{\widebar{K}^2_{vol_2}}{\sigma}.
        \end{equation}
    \end{enumerate}
\end{Corollary}
\begin{proof}
    Utilizing~\eqref{variance_z},~\eqref{variance_x}, and~\eqref{cov_x}, the variance of the log-return at time $t_N$ can be expressed as
    \begin{align*}
        \begin{split}
            \widebar{\sigma}_N &= \left( \frac{\sigma^2}{2\kappa}(1-e^{-2\kappa t_{N}}) +\frac{\sigma^2}{2\kappa}(1-e^{-2\kappa t_{N-1}}) -2\left( \frac{\sigma^2}{2\kappa}(1-e^{-2\kappa t_{N-1}})\right)e^{-\kappa \Delta t} \right)^\frac{1}{2} \\
            &= \sigma \left( \frac{1}{\kappa} + \frac{1}{2\kappa} \left( 2e^{-\kappa(t_{N} +t_{N-1} )} - 2e^{-\kappa(t_{N} -t_{N-1} )} -e^{-2 \kappa t_{N}}-e^{-2 \kappa t_{N-1}} \right) \right)^\frac{1}{2} .
        \end{split}
    \end{align*}
    We then obtain its partial derivative as
    \begin{equation} \label{diff_volatility_z}
        \frac{\partial \widebar{\sigma}_N}{\partial \sigma}=\frac{\widebar{\sigma}_N}{\sigma}.
    \end{equation}
    From Theorem~\ref{thm_Kvol2_1}, we consider the case that $\sum^{N-1}_{i=2} \widebar{\mu}_i^2 > 0$, we obtain the following partial derivative:
    \begin{align} \label{diff_Kvol2_1_3}
        \begin{split}
            \frac{\partial K^2_{vol_1}}{\partial \sigma} 
            &= \sqrt{\frac{\pi}{2 T}} \left( \widebar{\sigma}_N \frac{\partial}{\partial \sigma} \left( \mathbf{L}^{\left(\frac{\eta}{2}-1\right)}_{\frac{1}{2}} \! \left(-\sum^{N-1}_{i=2} \frac{\widebar{\mu}^2_i}{2\widebar{\sigma}^2_N} \right) \right) + \mathbf{L}^{\left(\frac{\eta}{2}-1\right)}_{\frac{1}{2}} \! \left(-\sum^{N-1}_{i=2} \frac{\widebar{\mu}^2_i}{2\widebar{\sigma}^2_N} \right) \left( \frac{\partial \widebar{\sigma}_N}{\partial \sigma} \right) \right) \times 100 \\
            &= \widebar{\sigma}_N \sqrt{\frac{\pi}{2 T}} \frac{\partial}{\partial \sigma} \left( \mathbf{L}^{\left(\frac{\eta}{2}-1\right)}_{\frac{1}{2}} \! \left(-\sum^{N-1}_{i=2} \frac{\widebar{\mu}^2_i}{2\widebar{\sigma}^2_N} \right) \right) \times 100  +  \frac{\widebar{K}^2_{vol_1}}{\sigma}. 
        \end{split}
    \end{align}
    By applying the property of the Laguerre function (see Mirevski and Boyadjiev~\cite{mirevski2010some}), we get
    \begin{equation}
        \frac{\partial}{\partial \sigma} \left( \mathbf{L}^{\left(\frac{\eta}{2}-1\right)}_{\frac{1}{2}} \left(-\sum^{N-1}_{i=2} \frac{\widebar{\mu}^2_i}{2\widebar{\sigma}^2_N} \right) \right) = \frac{\mathrm{d}}{\mathrm{d} \mathrm{u}} \mathbf{L}^{\left(\frac{\eta}{2}-1\right)}_{\frac{1}{2}} \left(\mathrm{-u} \right) \frac{\mathrm{d} }{\mathrm{d} \sigma} (-\mathrm{u}) = \mathbf{L}^{\frac{\eta}{2}}_{-\frac{1}{2}} \left(-\mathrm{u} \right) \frac{\mathrm{d} }{\mathrm{d} \sigma} (-\mathrm{u}),
        \label{diff_L_2}
    \end{equation}
    where $\mathrm{u}=\sum^{N-1}_{i=2} \frac{\widebar{\mu}^2_i}{2\widebar{\sigma}^2_N}$. Then the derivative of $\mathrm{u}$ with respect to $\widebar{\sigma}_N$ can be found as
    \begin{equation} \label{diff_lambda_2} 
        \frac{\mathrm{d} }{\mathrm{d} \sigma} (-\mathrm{u}) = - \frac{\mathrm{d} }{\mathrm{d} \sigma} (\mathrm{u}) = \frac{\sum^{N-1}_{i=2} \widebar{\mu}^2_i}{\widebar{\sigma}^3_N} \frac{\partial \widebar{\sigma}_N}{\partial \sigma} = \frac{\sum^{N-1}_{i=2} \widebar{\mu}^2_i}{\widebar{\sigma}^2_N} \left( \frac{1}{\sigma} \right) = \frac{\widebar{\lambda}}{\sigma}
    \end{equation} 
    The result shown in~\eqref{vega_Kvol2_1}, is simplified by substituting~\eqref{diff_lambda_2} into~\eqref{diff_L_2} and~\eqref{diff_Kvol2_1_3}. Under the condition that $\sum^{N-1}_{i=2} \widebar{\mu}_i^2 = 0$, by Theorem~\ref{thm_Kvol2_2}, it is easy to get that
    \begin{equation*}
        \frac{\partial K^2_{vol_2}}{\partial \sigma} = \frac{\partial K^2_{vol_2}}{\partial \widebar{\sigma}_N} \left( \frac{\partial \widebar{\sigma}_N}{\partial \sigma} \right) = \frac{\widebar{\sigma}_N}{\sigma} \sqrt{\frac{2}{T}} \frac{\Gamma \! \left(\frac{\eta+1}{2}\right)}{\Gamma \! \left(\frac{\eta}{2}\right)} \times 100,
    \end{equation*}
    which implies~\eqref{vega_Kvol2_2}.
\end{proof}

\begin{Corollary} \label{col_vega_Kvar2}
    According to Theorem~\ref{thm_Kvar2_1} and~\ref{thm_Kvar2_2}, we have
    \begin{equation} \label{vega_Kvar12}
        \mathlarger{\nu}_{K^2_{var_j}} =  \frac{2 \widebar{\sigma}^2_N }{\sigma T} \eta \times 100^2 = \frac{2\widebar{K}^2_{var_2}}{\sigma},
    \end{equation}
    for $j=1,2$ where $\eta$ is defined in~\eqref{degree_of_freedom_w}.
\end{Corollary}
\begin{proof}
    From Theorem~\ref{thm_Kvar2_1}, let $j=1$, if  $\sum^{N-1}_{i=2} \widebar{\mu}_i^2 > 0$, we arrive
    \begin{align*}
        \begin{split}
            \frac{\partial K^2_{var_1}}{\partial \sigma} &= \frac{1}{T} \left( 2 \eta \widebar{\sigma}_N \left( \frac{\partial \widebar{\sigma}_N}{\partial \sigma} \right) + \widebar{\sigma}^2_N \left( \frac{\partial \widebar{\lambda}}{\partial \sigma} \right) + 2 \widebar{\lambda} \widebar{\sigma}_N \left( \frac{\partial \widebar{\sigma}_N}{\partial \sigma} \right) \right)\times 100^2 \\
            &= \frac{2 \widebar{\sigma}^2_N }{\sigma T} \left( \eta  - \widebar{\lambda} +  \widebar{\lambda} \right)\times 100^2 \\
            &= \frac{2 \widebar{\sigma}^2_N }{\sigma T} \eta \times 100^2 ,
        \end{split}
    \end{align*}
    where $\frac{\partial \widebar{\sigma}_N}{\partial \sigma} = \frac{\widebar{\sigma}_N}{\sigma} \text{\ and \ }\frac{\partial \widebar{\lambda}}{\partial \sigma}=\frac{-2\widebar{\lambda}}{\sigma}$ can be obtained by using \eqref{diff_volatility_z} and applying \eqref{diff_lambda_2} with $\mathrm{u}=-\sum^{N-1}_{i=2} \frac{\widebar{\mu}^2_i}{\widebar{\sigma}^2_N}=-\widebar{\lambda}$, respectively. 
    
    On the other hand, let $j=1$, if  $\sum^{N-1}_{i=2} \widebar{\mu}_i^2 = 0$, from Theorem~\ref{thm_Kvar2_2}, it is trivial to see that this case yields the same result as the case before. We then obtain~\eqref{vega_Kvar12}
\end{proof}

In the case where an analytical formula for pricing volatility and variance swaps with time-varying log-return volatility, Theorems~\ref{thm_Kvol1_1} and~\ref{thm_Kvar1_1} are utilized. Finding a closed-form formula for vega in such cases remains a challenging topic for future work.


\section{Pricing volatility and variance options} \label{sec_pricing_option}
\medskip

As indicated in the put-call parity relations described by Hull in~\cite{hull2016options}, this section presents only an analytical formula for pricing volatility and variance options on calls.

A payoff for a volatility call buyer at expiration time $T$ on $[t_1, T]$ with a $K^{c}_{vol}$-strike can be written as 
\begin{equation} \label{payoff_vol_call}
    \big( \sqrt{RV_d(t_1, N, T)}-K^{c}_{vol} \big)^+
\end{equation}
and a payoff for a variance call buyer at expiration time $T$ on $[t_1, T]$ with a $K^{c}_{var}$-strike can be written as 
\begin{equation} \label{payoff_var_call}
    \big( RV_d(t_1, N, T)-K^{c}_{var} \big)^+
\end{equation}
for $t_1 \in [0,T)$ where $RV_d$ is the realized variance given in~\eqref{def_rv}.

Under the risk-neutral martingale measure $\mathbb{Q}$ with respect to filtration $\mathcal{F}_{t_1}$, we focus on calculating the conditional expectations of both call option yields given in~\eqref{payoff_vol_call} and~\eqref{payoff_var_call} above. In particular, volatility and variance calls can be expressed as
\begin{equation} \label{vol_call}
    C_{vol}:=e^{-\int_{t_1}^Tr(s) \ \mathrm{d}s} \mathbb{E}_{t_1}^\mathbb{Q}\big[ \big( \sqrt{RV_d(t_1, N, T)}-K^{c}_{vol} \big)^+\big],
\end{equation}
where $K^{c}_{vol}$ is a strike volatility call and
\begin{equation} \label{var_call}
    C_{var}:=e^{-\int_{t_1}^Tr(s) \ \mathrm{d}s} \mathbb{E}_{t_1}^\mathbb{Q}\big[ \big( RV_d(t_1, N, T)-K^{c}_{var} \big)^+ \big],
\end{equation}
where $K^{c}_{var}$ is a strike variance call and $r(t)$ is a time-varying risk-free interest rate, respectively. The challenge of this part is that we need to compute both expectations on the RHS of calls~\eqref{vol_call} and~\eqref{var_call}, given as follows:
\begin{equation} \label{int_vol_call}
    \mathbb{E}_{t_1}^\mathbb{Q}\big[ \big(\sqrt{RV_d(t_1, N, T)}-K^{c}_{vol} \big)^+ \big] = \int_{K^{c}_{vol}}^\infty (y-K^{c}_{vol})\,2y\widebar{f}^{(\widebar{\beta}, \widebar{\mu}_0)}_n (y^2) \, \mathrm{d} y,
\end{equation}
and
\begin{equation} \label{int_var_call}
    \mathbb{E}_{t_1}^\mathbb{Q}\big[ \big(RV_d(t_1, N, T)-K^{c}_{vol} \big)^+ \big] = \int_{K^{c}_{var}}^\infty (y-K^{c}_{var})\widebar{f}^{(\widebar{\beta}, \widebar{\mu}_0)}_n (y) \, \mathrm{d} y,
\end{equation}
respectively, where $\widebar{f}^{(\widebar{\beta}, \widebar{\mu}_0)}_n (y)$ is the PDF given in~\eqref{pdf_rv} from Theorem~\ref{thm_pdf_rv}. To obtain the RHS of~\eqref{int_vol_call}, we apply a Jacobian transformation, which gives the relation $\widebar{f}^{(\widebar{\beta}, \widebar{\mu}_0)}_n (\sqrt{y})=2y\widebar{f}^{(\widebar{\beta}, \widebar{\mu}_0)}_n (y^2)$.

As the same problem with Section 3, to avoid the complication of numerical integration for the improper integral in~\eqref{int_var_call} and~\eqref{int_vol_call}, both integrals can be expressed as an infinite summation series through a generalized Laguerre expansion derived by Dufresne~\cite{dufresne2000laguerre} in Theorem 2.4. The conditional expectations can then be simplified as
\begin{align} \label{expected_call}
    \mathbb{E}_{t_1}^\mathbb{Q}\big[\,(RV_d(t_1, N, T)^{\rho} -K)^+\big] :=K^b e^{-K} \sum_{k=0}^\infty h_k (\rho, t_1, N, T) \, \mathbf{L}_k^{(a)} (K),
\end{align} 
for $\rho=1/2,1$ where $a$ and $b$ are real numbers satisfying $a > 2\max(b,0)-1$. The coefficient $h_k (\rho,t_1, N, T)$ on the RHS of~\eqref{expected_call} is defined in terms of a finite summation series of the conditional moment of $RV_d(t_1, N, T)$ order $\rho \tau_j$ for $j=0,1,2,\ldots$, as follows: 
\begin{equation} \label{coefficient_hk}
    h_k (\rho,t_1, N, T):=\sum_{j=0}^k \frac{k! \, (-1)^j \, \mathbb{E}_{t_1}^\mathbb{Q}\big[\,RV_d(t_1, N, T)^{\rho \tau_j}\big]}{\Gamma(j+a+1) \, j! \, (k-j)! \, (\rho \tau_j-1) \, \rho \tau_j},
\end{equation}
for $j=0,1,2,\ldots$, where $\rho\tau_j= a-b+j+2$ where $\mathbb{E}_{t_1}^\mathbb{Q}\big[\,RV_d(t_1, N, T)^{\rho \tau_j}\big]$ can be specified based on call options' type. By setting $\rho = 1/2$, we obtain the coefficient $h_k$ for the volatility call, and for the variance call, we set $\rho = 1$. 

In the following subsection, an analytical formula for pricing volatility and variance options is presented in two cases by setting the variance of the log-return: one with time-varying variances and the other with a constant variance.
 

\subsection{Pricing formulas for volatility and variance options: Time-varying log-return volatility} \label{sec_C_vary}
\medskip 

In this subsection, the variances of log-returns are assumed to change over time, similar to the setup introduced in Subsection~\ref{sec_K_constant}. Then, $RV_d(t_1,N,T)$ is considered as derived in~\eqref{rv}. Its conditional moment can be computed by using Theorem~\ref{thm_moment} to evaluate the RHS of~\eqref{coefficient_hk} and complete the coefficient $h_k(\rho,t_1, N, T)$.


\subsubsection{Volatility call option prices: \texorpdfstring{$C_{vol}^1$}{}}
\medskip 

By using the PDF~\eqref{pdf_rv}, the following theorem provides a closed-form formula for pricing
a volatility call option under time-varying log-return volatility.

\begin{Theorem} \label{thm_call1_vol}
    Suppose that $\rho=\frac{1}{2}$. The volatility call of a strike $K^c_{vol}$ at time $t_1$ can be expressed as
    \begin{equation} \label{callSwap1_vol}
        C_{vol}^1 \equiv C_{vol} \left( \frac{1}{2},t_1, N, T \right) =e^{-\int_{t_1}^Tr(s) \, \mathrm{d} s} (K^c_{vol})^b e^{-K^c_{vol}} \sum_{k=0}^\infty h_k \left( \frac{1}{2}, t_1, N, T \right) \mathbf{L}_k^{(a)} (K^c_{vol}),
    \end{equation}
    where $h_k \left( \frac{1}{2}, t_1, N, T \right)$, $k=0,1,2,\ldots$, are calculated by using~\eqref{coefficient_hk} with $\mathbb{E}_{t_1}^\mathbb{Q}\big[\,RV_d(t_1, N, T) ^{\frac{\tau_j}{2}}\big]$ written in Theorem~\ref{thm_moment} where $\frac{\tau_j}{2}= a-b+j+2$ for $j=0,1,2,\ldots$ with scalars $a, b \in \mathbb{R}$ satisfying $a > 2\max(b,0)-1$.
\end{Theorem}
\begin{proof}
    A volatility call can be found by using~\ref{vol_call} together with its conditional expectation on the RHS can be computed by setting $K=K^c_{vol}$ and $\rho=1/2$ in~\eqref{expected_call} with coefficients $h_k$ for $k=0,1,2,\ldots$, given in~\eqref{coefficient_hk}. To obtain~\eqref{coefficient_hk}, we have $\ell=\tau_j/2$ for $j=0,1,2,\ldots$, which is set in~\eqref{generalmoment}. We then immediately obtain~\eqref{callSwap1_vol}.
\end{proof}


\subsubsection{Variance call option prices: \texorpdfstring{$C_{var}^1$}{}}
\medskip 

The following theorem is a closed-form formula for pricing a variance call option under time-varying log-return volatility.

\begin{Theorem} \label{thm_call1_var}
    Suppose that $\rho=1$. The variance call of a strike $K^c_{var}$ at time $t_1$ can be expressed as
    \begin{equation} \label{callSwap1_var}
        C_{var}^1 \equiv C_{var} \left( 1,t_1, N, T \right) =e^{-\int_{t_1}^Tr(s) \, \mathrm{d} s} (K^c_{var})^b e^{-K^c_{var}} \sum_{k=0}^\infty h_k \left( 1, t_1, N, T \right) \mathbf{L}_k^{(a)} (K^c_{var}),
    \end{equation}
    where $h_k \left( 1, t_1, N, T \right)$, $k=0,1,2,\ldots$, are calculated by using~\eqref{coefficient_hk} with $\mathbb{E}_{t_1}^\mathbb{Q}\big[ \, RV_d(t_1, N, T)^{\tau_j}\big]$ written in Theorem~\ref{thm_moment} where $\tau_j= a-b+j+2$ for $j=0,1,2,\ldots$ with scalars $a, b \in \mathbb{R}$ satisfying $a > 2\max(b,0)-1$.
\end{Theorem}
\begin{proof}
    We get a variance call by using~\eqref{var_call} and we get the conditional expectation on the RHS of~\eqref{var_call} by replacing $K=K^c_{var}$ and $\rho=1$ in~\eqref{expected_call}. On the RHS of~\eqref{expected_call}, coefficients $h_k$ for $k=0,1,2,\ldots$, defined in~\eqref{coefficient_hk}, are computed by letting $\ell=\tau_j$ for $j=0,1,2,\ldots$, in~\eqref{generalmoment}. Hence, we immediately obtain~\eqref{callSwap1_var}.
\end{proof}


\subsection{Pricing formulas for volatility and variance options: Constant log-return volatility} \label{sec_C_constant}
\medskip 

The interesting results of this subsection are other explicit solutions of volatility and variance call options, obtained by assuming that the variance of log-returns remains the same. Specifically, we set $\widebar{\sigma}^2_2 = \widebar{\sigma}^2_3 = \dots = \widebar{\sigma}^2_N$, following the approach introduced in Subsection~\ref{sec_K_constant}. Under this assumption, the realized variance is expressed as a linear combination of independent noncentral chi-square random variables, and its conditional moment is derived below.

Since $\widebar{W} \sim \mathcal{NC}_{\chi_{\eta}^2}\left(\widebar{\lambda}\right)$ defined in~\eqref{w}, its conditional moment order $\ell$, for $\ell > 0$, is obtained from Stuart et al.~\cite{stuart1999kendall} in explicit form as
\begin{equation} \label{generalmoment_w}
     \mathbb{E}_{t_1}
    ^\mathbb{Q}[\widebar{W}^{\ell}] = 2^{\ell} e^{\frac{-\widebar{\lambda}}{2}} \, \frac{\Gamma \! \left( \ell+ \frac{\eta}{2} \right)}{\Gamma \! \left(\frac{\eta}{2} \right)} \, {_1}\mathbf{F}_1 \! \! \left( \ell+ \frac{\eta}{2}; \frac{\eta}{2}; \frac{\widebar{\lambda}}{2} \right),
\end{equation}
where degrees of freedom $\eta$ and noncentrality parameter $\widebar{\lambda}$ are given in~\eqref{degree_of_freedom_w} and~\eqref{noncentrality_w}, respectively. Hence, a conditional moment of $RV_d(t_1, N, T)$ order $\ell$, for $\ell > 0$, is defined as
\begin{equation} \label{generalmoment_rv_noncentral}
    \mathbb{E}_{t_1}^\mathbb{Q}[RV_d(t_1, N, T)^{\ell}]:=\left( \frac{\widebar{\sigma}_N^2}{T} \times 100 \right)^{\ell} \! \mathbb{E}_{t_1}
    ^\mathbb{Q}\big[ \, \widebar{W}^{\ell}\big].
\end{equation}

For $\widebar{\lambda}=0$, $\widebar{W}$ is distributed according to the central chi-square distribution with degrees of freedom $\eta$. By using a property of the generalized hypergeometric function~\eqref{hyper_fun}, it is clear that a conditional moment of $RV_d(t_1, N, T)$ order $\ell$ for $\ell>0$, can be simplified as
\begin{equation} \label{generalmoment_rv_central}
    \mathbb{E}_{t_1}^\mathbb{Q}\big[ \, RV_d(t_1, N, T)^{\ell} \big]:=\left( \frac{2\widebar{\sigma}_N^2}{T} \times 100 \right)^{\ell} \frac{\Gamma \! \left( \ell + \frac{\eta}{2} \right)}{\Gamma \! \left(\frac{\eta}{2} \right)}.
\end{equation}


\subsubsection{Volatility call option prices: \texorpdfstring{$C_{vol}^2$}{}}
\medskip

Since $\sqrt{RV_d(t_1,N,T)}$ can be represented in terms of $\sqrt{\widebar{W}}$, the following theorem provides a closed-form formula for pricing a call volatility option under constant log-return volatility by applying~\eqref{generalmoment_w}.

\begin{Theorem} \label{thm_call2_vol}
    Suppose that $\rho=\frac{1}{2}$. The volatility call of a strike $K^c_{vol}$ at time $t_1$ can be expressed as
    \begin{equation} \label{callSwap2_vol}
        C_{vol}^2 \equiv C_{vol} \left( \frac{1}{2},t_1, N, T \right) =e^{-\int_{t_1}^Tr(s) \, \mathrm{d} s} (K^c_{vol})^b e^{-K^c_{vol}} \sum_{k=0}^\infty h_k \left( \frac{1}{2}, t_1, N, T \right) \mathbf{L}_k^{(a)} (K^c_{vol}),
    \end{equation}
    where $h_k \left( \frac{1}{2}, t_1, N, T \right)$, $k=0,1,2,\ldots$, are calculated by using~\eqref{coefficient_hk} with the following two conditions holding:
    \begin{enumerate}
        \item if \ \ $\displaystyle \sum^{N-1}_{i=2} \widebar{\mu}_i^2 > 0$ \ \ then \ \ $\mathbb{E}_{t_1}^\mathbb{Q}\big[ \, RV_d(t_1, N, T) ^{\frac{\tau_j}{2}} \big]$ is calculated by using~\eqref{generalmoment_rv_central}
          \vspace{2mm}
        \item if \ \ $\displaystyle \sum^{N-1}_{i=2} \widebar{\mu}_i^2 = 0$ \ \ then \ \ $\mathbb{E}_{t_1}^\mathbb{Q} \big[ \, RV_d(t_1, N, T) ^{\frac{\tau_j}{2}} \big]$ is calculated by using~\eqref{generalmoment_rv_noncentral}
    \end{enumerate}
    where $\frac{\tau_j}/{2}= a-b+j+2$ for $j=0,1,2,\ldots$, with scalars $a, b \in \mathbb{R}$ satisfying $a > 2\max(b,0)-1$.
\end{Theorem}

\begin{proof}
    By using~\eqref{vol_call}, we obtain its conditional expectation on the RHS of~\eqref{vol_call} by setting $K=K^c_{vol}$ and $\rho=1/2$ in~\eqref{expected_call} with coefficients $h_k$ for $k=0,1,2,\ldots$, defined in~\eqref{coefficient_hk}. To obtain~\eqref{coefficient_hk}, if $\sum^{N-1}_{i=2} \widebar{\mu}_i^2 > 0$, $\widebar{\lambda} > 0$, then we calculate the conditional moment of $RV_d(t_1,N,T)$ of order $\ell$ by setting  $\ell=\tau_j/2$ for $j=0,1,2,\ldots$, in~\eqref{generalmoment_rv_noncentral}. On the other hand, for the case that $\sum^{N-1}_{i=2} \widebar{\mu}_i^2 = 0$, which implies $\widebar{\lambda} = 0$, we calculate it using~\eqref{generalmoment_rv_central}. We then obtain the result in~\eqref{callSwap2_vol}.
\end{proof}


\subsubsection{Variance call option prices: \texorpdfstring{$C_{var}^2$}{}}
\medskip

The following theorem is a closed-form formula for pricing a variance call option under constant log-return volatility by applying~\eqref{generalmoment_w}.

\begin{Theorem} \label{thm_call2_var}
    Suppose that $\rho=1$. The volatility call of a strike $K^c_{var}$ at time $t_1$ can be expressed as
    \begin{equation} \label{callSwap2_var}
        C_{var}^2 \equiv C_{var} \left( \frac{1}{2},t_1, N, T \right) =e^{-\int_{t_1}^Tr(s) \, \mathrm{d} s} (K^c_{var})^b e^{-K^c_{vol}} \sum_{k=0}^\infty h_k \left( \frac{1}{2}, t_1, N, T \right) \mathbf{L}_k^{(a)} (K^c_{var}),
    \end{equation}
    where $h_k \left( 1, t_1, N, T \right)$, $k=0,1,2,\ldots$, are calculated by using~\eqref{coefficient_hk} with the following two conditions holding: 
    \begin{enumerate}
      \item if \ \ $\displaystyle \sum^{N-1}_{i=2} \widebar{\mu}_i^2 > 0$ \ \ then \ \ $\mathbb{E}_{t_1}^\mathbb{Q} \big[ \, RV_d(t_1, N, T) ^{\tau_j} \big]$ \ \ is calculated by using~\eqref{generalmoment_rv_noncentral},
      \vspace{2mm}
      \item if \ \ $\displaystyle \sum^{N-1}_{i=2} \widebar{\mu}_i^2 = 0$ \ \ then \ \ $\mathbb{E}_{t_1}^\mathbb{Q}\big[ \, RV_d(t_1, N, T) ^{\tau_j} \big]$ is calculated by using~\eqref{generalmoment_rv_central},
    \end{enumerate}
    where $\tau_j= a-b+j+2$ for $j=0,1,2,\ldots$, with scalars $a, b \in \mathbb{R}$ satisfying $a > 2\max(b,0)-1$.
\end{Theorem}
\begin{proof}
    We get a variance call by using~\eqref{var_call} and its conditional expectation on the RHS of~\eqref{var_call} by setting $K=K^c_{var}$ and $\rho=1$ in~\eqref{expected_call} with coefficients $h_k$ for $k=0,1,2,\ldots$, defined in~\eqref{coefficient_hk}. To obtain~\eqref{coefficient_hk}, we proceed with the same approach in the proof of the previous Theorem~\ref{thm_call2_vol}. But now we let $\ell=\tau_j$. We then obtain the result in~\eqref{callSwap2_var}.
\end{proof}


\subsubsection{Vega of \texorpdfstring{$C_{vol}^2$}{} and \texorpdfstring{$C_{var}^2$}{}}
\medskip

To obtain vega of call options presented in Theorems~\ref{thm_call2_vol} and~\ref{thm_call2_var}, we focus on its partial derivative on the RHS with respect to $\sigma$. We have that the parameter $\sigma$ appears only in $\widebar{\sigma}_N$, where $\widebar{\sigma}_N$ is an argument in $\mathbb{E}_{t_1}^\mathbb{Q}\big[ \, RV_d(t_1, N, T)^\ell \big]$, as defined in Subsection~\ref{sec_C_constant}. For a given $\ell=\rho \tau_j$ for $\rho=1/2,1$ and $j=0,1,2,\ldots$, we have coefficients $h_k$ for $k=0,1,2,\ldots$, as defined in~\eqref{coefficient_hk}. Then, the partial derivatives of $C^2_{vol}$ and $C^2_{var}$ can be expressed as an infinite summation series of $h'_k$, which $h'_k$ is defined as
\begin{equation} \label{diff_coefficient_hk}
    h'_k (\rho,t_1, N, T) :=\frac{\partial}{\partial \sigma} h_k (\rho,t_1, N, T) = \sum_{j=0}^k \frac{k! \, (-1)^j \, \frac{\partial}{\partial \sigma}  \mathbb{E}_{t_1}^\mathbb{Q}\big[ \, RV_d(t_1, N, T)^{\rho \tau_j}\big] }{\Gamma(j+a+1) \, j! \, (k-j)! \, (\rho \tau_j-1) \, \rho \tau_j}. 
\end{equation}

In the following theorems, we present $\frac{\partial}{\partial \sigma}  \mathbb{E}_{t_1}^\mathbb{Q} \big[ \, RV_d(t_1, N, T)^\ell \big]$ for two cases, characterized by the value of $\widebar{\lambda}$. 

\begin{Theorem} \label{thm_diff_generalmoment_rv_noncentral}
    According to~\eqref{generalmoment_rv_noncentral}, if \ $\displaystyle \sum^{N-1}_{i=2} \widebar{\mu}_i^2 > 0$ , then the partial derivative of the conditional moment of $RV_d(t_1, N, T)$ order $\ell$ for $\ell > 0$, is defined as
    \begin{align} \label{diff_moment_rv_noncentral}
        \begin{split}
        &\frac{\partial}{\partial \sigma} \mathbb{E}_{t_1}^\mathbb{Q}\big[ \, RV_d(t_1, N, T)^{\ell} \big] \\
        &= \frac{e^{\frac{-\widebar{\lambda}}{2}}}{\sigma} \left( \frac{2 \widebar{\sigma}_N^{2}}{T} \times 100 \right)^{\ell} 
        \frac{\Gamma \! \left( \ell+ \frac{\eta}{2} \right)}{\Gamma \!  \left(\frac{\eta}{2} \right)} 
        \left[ 
        \left( \widebar{\lambda} + 2\ell\right) {_1}\mathbf{F}_1 \! \! \left( \ell+ \frac{\eta}{2}; \frac{\eta}{2}; \frac{\widebar{\lambda}}{2} \right)
        -\widebar{\lambda} \frac{\left( \ell+ \frac{\eta}{2} \right)_1}{\left( \frac{ \eta}{2} \right)_1}{_1}\mathbf{F}_1 \! \! \left( 1+\ell+ \frac{\eta}{2}; 1+\frac{\eta}{2}; \frac{\widebar{\lambda}}{2} \right) 
        \right] \\
        &= \frac{\left( \widebar{\lambda} + 2\ell\right)}{\sigma} \mathbb{E}_{t_1}^\mathbb{Q}\big[ \, RV_d(t_1, N, T)^{\ell}\big] -\left( \frac{2 \widebar{\sigma}_N^{2}}{T} \times 100 \right)^{\ell} \frac{\widebar{\lambda} \, e^{\frac{-\widebar{\lambda}}{2}}}{\sigma} \frac{\Gamma \! \left( 1+\ell+ \frac{\eta}{2} \right)}{\Gamma \! \left(1+ \frac{ \eta}{2} \right)}{_1}\mathbf{F}_1 \! \! \left( 1+\ell+ \frac{\eta}{2}; 1+\frac{\eta}{2}; \frac{\widebar{\lambda}}{2} \right).
        \end{split}
    \end{align}
\end{Theorem}
\begin{proof}
    Suppose that $\sum^{N-1}_{i=2} \widebar{\mu}_i^2 > 0$. Then, there exists $\widebar{\lambda} > 0$ such that the partial derivative of the conditional moment of order $\ell$, given in~\eqref{generalmoment_rv_noncentral}, can be expressed as 
    \begin{align} \label{diff_generalmoment_rv_noncentral}
        \begin{split}
        &\frac{\partial}{\partial \sigma} \mathbb{E}_{t_1}^\mathbb{Q}\big[ \, RV_d(t_1, N, T)^{\ell}\big] \\
        &= \left( \frac{2}{T} \times 100 \right)^{\ell} \frac{\Gamma \! \left( \ell+ \frac{\eta}{2} \right)}{\Gamma \! \left(\frac{\eta}{2} \right)} 
        \left[
        \widebar{\sigma}_N^{2\ell} \, e^{\frac{-\widebar{\lambda}}{2}}  \frac{\partial}{\partial \sigma} {_1}\mathbf{F}_1 \! \! \left( \ell+ \frac{\eta}{2}; \frac{\eta}{2}; \frac{\widebar{\lambda}}{2} \right) 
        + {_1}\mathbf{F}_1 \! \! \left( \ell+ \frac{\eta}{2}; \frac{\eta}{2}; \frac{\widebar{\lambda}}{2} \right) \frac{\partial}{\partial \sigma} \left( \widebar{\sigma}_N^{2\ell} \, e^{\frac{-\widebar{\lambda}}{2}} \right) \right].
        \end{split}
    \end{align}
    From the above equation~\eqref{diff_generalmoment_rv_noncentral}), firstly, we consider the partial derivative of the hypergeometric function, by using its differentiation formula, see~\cite[Chapter 13]{olver2010nist}, we have
    \begin{align} \label{diff_1F1}
        \begin{split}
            \frac{\partial}{\partial \sigma} {_1}\mathbf{F}_1 \! \! \left( \ell+ \frac{\eta}{2}; \frac{\eta}{2}; \frac{\widebar{\lambda}}{2} \right) 
            &=\frac{\partial}{\partial \sigma} {_1}\mathbf{F}_1 \! \! \left( \ell+ \frac{\eta}{2}; \frac{\eta}{2}; \sum^{N-1}_{i=2} \frac{\widebar{\mu}^2_i}{2\widebar{\sigma}^2_N} \right) \\
            &= \frac{\partial}{\partial \sigma} {_1}\mathbf{F}_1 \! \! \left( \ell+ \frac{\eta}{2}; \frac{\eta}{2}; \mathrm{u} \right) \frac{\mathrm{d} \mathrm{u}}{\mathrm{d} \sigma} \\
            &= \frac{\left( \ell+ \frac{\eta}{2} \right)_1}{\left(\frac{\eta}{2} \right)_1} {_1}\mathbf{F}_1 \! \! \left( 1+\ell+ \frac{\eta}{2}; 1+\frac{\eta}{2}; \mathrm{u} \right) \frac{\mathrm{d} \mathrm{u}}{\mathrm{d} \sigma},
        \end{split}
    \end{align}
    where $\mathrm{u}=\sum^{N-1}_{i=2} \frac{\widebar{\mu}^2_i}{2\widebar{\sigma}^2_N}$. By applying~\eqref{diff_lambda_2}, we get $\frac{\mathrm{d} \mathrm{u}}{\mathrm{d} \sigma} = -\frac{\widebar{\lambda}}{\sigma}$. Secondly, we consider the partial derivative of the second term on the RHS of~\eqref{diff_generalmoment_rv_noncentral}, we have 
    \begin{equation} \label{diff_product}
        \frac{\partial}{\partial \sigma} \left( \widebar{\sigma}_N^{2\ell} e^{\frac{-\widebar{\lambda}}{2}} \right) = \widebar{\sigma}_N^{2\ell} \left( \frac{\partial}{\partial \sigma} e^{\frac{-\widebar{\lambda}}{2}} \right) + e^{\frac{-\widebar{\lambda}}{2}} \left( \frac{\partial}{\partial \sigma} \widebar{\sigma}_N^{2\ell} \right),
    \end{equation}
    where 
    \begin{equation} \label{diff_e}
        \frac{\partial}{\partial \sigma} e^{\frac{-\widebar{\lambda}}{2}} = \frac{\partial}{\partial \sigma} \exp{\left(-\sum^{N-1}_{i=2} \frac{\widebar{\mu}^2_i}{2\widebar{\sigma}^2_N}\right)} = \frac{\partial}{\partial \sigma} e^{- \mathrm{u}} = e^{- \mathrm{u}} \frac{\mathrm{d} }{\mathrm{d} \sigma} (\mathrm{-u}) = \frac{\widebar{\lambda}}{\sigma} e^{\frac{-\widebar{\lambda}}{2}},
    \end{equation}
    where $\frac{\mathrm{d} }{\mathrm{d} \sigma} (\mathrm{-u}) = \frac{\widebar{\lambda}}{\sigma}$ is derived in~\eqref{diff_lambda_2}, and
    \begin{equation} \label{diff_variance_z}
        \frac{\partial}{\partial \sigma} \widebar{\sigma}_N^{2\ell} = 2\ell \, \widebar{\sigma}_N^{2\ell-1} \left( \frac{\partial \widebar{\sigma}_N}{\partial \sigma} \right) =\frac{2\ell \, \widebar{\sigma}_N^{2\ell}}{\sigma},
    \end{equation}
    where $\frac{\partial \widebar{\sigma}_N}{\partial \sigma} = \frac{\widebar{\sigma}_N}{\sigma}$ is derived in~\eqref{diff_volatility_z}. Thus, the result~\eqref{diff_moment_rv_noncentral} can be simplified by substituting~\eqref{diff_e} and (\eqref{diff_variance_z}) into~\eqref{diff_product} which~\eqref{diff_1F1} and~\eqref{diff_product} are substituted into~\eqref{diff_generalmoment_rv_noncentral}.
\end{proof}

\begin{Theorem} \label{thm_diff_generalmoment_rv_central}
    According to~\eqref{generalmoment_rv_noncentral}, if \ $\displaystyle \sum^{N-1}_{i=2} \widebar{\mu}_i^2 = 0$, then the partial derivative of the conditional moment of $RV_d(t_1, N, T)$ order $\ell$ for $\ell > 0$ is defined as
    \begin{equation} \label{diff_moment_rv_central}
        \frac{\partial}{\partial \sigma} \mathbb{E}_{t_1}^\mathbb{Q}\big[ \, RV_d(t_1, N, T)^{\ell}\big]
        = \frac{2\ell}{\sigma} \left( \frac{2 \widebar{\sigma}_N^2}{T} \times 100 \right)^{\ell} \frac{\Gamma \! \left( \ell+ \frac{\eta}{2} \right)}{\Gamma \! \left(\frac{\eta}{2} \right)} = \frac{2\ell}{\sigma} \mathbb{E}_{t_1}^\mathbb{Q}\big[ \, RV_d(t_1, N, T)^{\ell}\big]. 
    \end{equation}
\end{Theorem}
\begin{proof}
    Suppose that $\sum^{N-1}_{i=2} \widebar{\mu}_i^2 = 0$, then $\widebar{\lambda}=0$. From the conditional moment given in~\eqref{generalmoment_rv_central}, it is easy to obtain its partial derivative that
    \begin{equation*} \label{diff_generalmoment_rv_central}
        \frac{\partial}{\partial \sigma} \mathbb{E}_{t_1}^\mathbb{Q}\big[ \, RV_d(t_1, N, T)^{\ell}\big]
        = \left( \frac{2}{T} \times 100 \right)^{\ell} \frac{\Gamma \! \left( \ell+ \frac{\eta}{2} \right)}{\Gamma \! \left(\frac{\eta}{2} \right)} \frac{\partial}{\partial \sigma} \widebar{\sigma}_N^{2\ell}.
    \end{equation*}
    By substituting $\frac{\partial}{\partial \sigma} \widebar{\sigma}_N^{2\ell} = \frac{2\ell \widebar{\sigma}_N^{2\ell}}{\sigma}$, as derived in~\eqref{diff_variance_z} from the proof of the previous Theorem~\ref{thm_diff_generalmoment_rv_noncentral}, we immediately obtain~\eqref{diff_moment_rv_central}.
\end{proof}

According to vega defined in the third-level Section~\ref{sec_vega_k}, the derivative instruments are considered as volatility and variance call options. In the following corollaries, exact formulas for vega of volatility and variance call options,   $\mathcal{V}_{C^2_{vol}}$ and $\mathcal{V}_{C^2_{vol}}$, are derived by using Theorems~\ref{thm_call2_vol} and~\ref{thm_call2_var}, respectively.     

\begin{Corollary} \label{col_vega_call2_vol}
    According to Theorem~\ref{thm_call2_vol}, we have
    \begin{equation} \label{vega_call2_vol}
        \mathcal{V}_{C^2_{vol}}=e^{-\int_{t_1}^Tr(s) \, \mathrm{d} s} (K^c_{vol})^b e^{-K^c_{vol}} \sum_{k=0}^\infty h'_k \! \left( \frac{1}{2}, t_1, N, T \right) \mathbf{L}_k^{(a)} (K^c_{vol}),
    \end{equation}
    where $h'_k \left( \frac{1}{2}, t_1, N, T \right)$, $k=0,1,2,\ldots$, are calculated by using~\eqref{diff_coefficient_hk} with the following two conditions holding:
    \begin{enumerate}
      \item if \ \ $\displaystyle \sum^{N-1}_{i=2} \widebar{\mu}_i^2 > 0$ \ \ then \ \ $\frac{\partial}{\partial \sigma}\mathbb{E}_{t_1}^\mathbb{Q}\big[ \, RV_d(t_1, N, T) ^{\frac{\tau_j}{2}} \big]$ \ \ is calculated by using Theorem~\ref{thm_diff_generalmoment_rv_noncentral}
      \vspace{2mm}
      \item if \ \ $\displaystyle \sum^{N-1}_{i=2} \widebar{\mu}_i^2 = 0$ \ \ \ then \ \ $\frac{\partial}{\partial \sigma}\mathbb{E}_{t_1}^\mathbb{Q}\big[ \, RV_d(t_1, N, T) ^{\frac{\tau_j}{2}}\big]$ \ \ is calculated by using Theorem~\ref{thm_diff_generalmoment_rv_central} ,
    \end{enumerate}
    where $\frac{\tau_j}{2}= a-b+j+2$ for $j=0,1,2,\ldots$, with scalars $a, b \in \mathbb{R}$ satisfying $a > 2\max(b,0)-1$.
\end{Corollary}
\begin{proof}
    Due to Theorem~\ref{thm_call2_vol}, the partial derivative of $C^2_{vol}$ with respect to $\sigma$ can be illustrated as an infinite summation series of coefficients $h'_k$, $k=0,1,2,\ldots$, as defined in~\eqref{diff_coefficient_hk}. For a given $\rho=1/2$ in~\eqref{diff_coefficient_hk}, if $\sum^{N-1}_{i=2} \widebar{\mu}_i^2 > 0$, we compute $\frac{\partial}{\partial \sigma} \mathbb{E}_{t_1}^\mathbb{Q}\big[ \, RV_d(t_1, N, T)^\ell\big]$ using Theorem~\ref{thm_diff_generalmoment_rv_noncentral}. On the other hand, if $\sum^{N-1}_{i=2} \widebar{\mu}_i^2 = 0$, we compute it using Theorem~\ref{thm_diff_generalmoment_rv_central}. Finally, we obtain a result in~\eqref{vega_call2_vol} by setting $\ell=\tau_j$ for $j=0,1,2,\ldots$.
\end{proof}

\begin{Corollary} \label{col_vega_call2_var}
    According to Theorem~\ref{thm_call2_var}, we have
    \begin{equation} \label{vega_call2_var}
        \mathcal{V}_{C^2_{var}}=e^{-\int_{t_1}^Tr(s) \, \mathrm{d} s} (K^c_{vol})^b e^{-K^c_{vol}} \sum_{k=0}^\infty h'_k \left( 1, t_1, N, T \right) \, \mathbf{L}_k^{(a)} (K^c_{vol}),
    \end{equation}
    where $h'_k \left( 1, t_1, N, T \right)$, $k=0,1,2,\ldots$, are calculated by using~\eqref{diff_coefficient_hk} with the following two conditions holding:
    \begin{enumerate}
      \item if \ \ $\displaystyle \sum^{N-1}_{i=2} \widebar{\mu}_i^2 > 0$ \ \ then \ \ $\frac{\partial}{\partial \sigma}\mathbb{E}_{t_1}^\mathbb{Q}\big[ \, RV_d(t_1, N, T) ^{\tau_j}\big]$ \ \ is calculated by using Theorem~\ref{thm_diff_generalmoment_rv_noncentral}
      \vspace{2mm}
      \item if \ \ $\displaystyle \sum^{N-1}_{i=2} \widebar{\mu}_i^2 = 0$ \ \ then \ \ $\frac{\partial}{\partial \sigma}\mathbb{E}_{t_1}^\mathbb{Q}\big[ \, RV_d(t_1, N, T) ^{\tau_j}\big]$ \ \ is calculated by using Theorem~\ref{thm_diff_generalmoment_rv_central} ,
    \end{enumerate}
    where $\tau_j= a-b+j+2$ for $j=0,1,2,\ldots$, with scalars $a, b \in \mathbb{R}$ satisfying $a > 2\max(b,0)-1$.
\end{Corollary}
\begin{proof}
    It is easy to obtain~\eqref{vega_call2_var} by following the previous proof of Corollary~\ref{col_vega_call2_vol} but now the parameter $\ell$ is replaced by $1$.
\end{proof}


\section{Numerical results and discussion} \label{sec_application}
\medskip

As demonstrated in Section~\ref{sec_pdf_rv}, the theoretical frameworks developed in this paper yield new explicit formulas for computing the PDF of $RV_d$, as defined in equation~\eqref{def_rv}, along with its conditional moments. This includes an analytical formula for pricing volatility and variance swaps and options under Schwartz one-factor model, as derived in Section~\ref{sec_pricing_swaps} and~\ref{sec_pricing_option} respectively. In practice, we automatically question whether these newly derived explicit formulas are accurate and efficient, especially considering that an infinite sum has to be truncated. To ensure that there are no algebraic errors in the derivation process, we thoroughly examine the accuracy of our explicit formulas. Additionally, to demonstrate the efficiency of our explicit formulas compared to MC simulations, we conduct a series of numerical examples coded in MATHEMATICA 13 and MATLAB R2024b, performed on a computer notebook with the following specifications: Processor: 2 GHz Quad-Core Intel Core i5, Memory: 16 GB 3733 MHz LPDDR4X, Operating system: macOS 14.2.1 (23C71).


\subsection{Accuracy of the PDF of realized variance}
\medskip

To show the accuracy for our explicit formula of the PDF of $RV_d$, denoted by $\widebar{f}_\nu^{(\widebar{\beta},\widebar{\mu}_0)} (y)$, as defined in~\eqref{pdf_rv}, we choose $\widebar{\mu}_0=\nu/2$ and then $\widebar{\beta}> (1/2) \max_{i \in \{2,\ldots,N\}} \widebar{\alpha}_i$. We set $\widebar{\beta}=\max_{i \in \{2,\ldots,N\}} \widebar{\alpha}_i$ to hold for all our numerical experiments.

\begin{Example} \label{ex.1}
    In this example, we illustrate the shapes of our explicit formula for the PDFs of $RV_d$ by plotting $\sum_{i=2}^{N} \widebar{\alpha}_i \widebar{Y}_i$, which are varied by the number of observations $N$.
\end{Example}

By setting $T=1$, the shapes of the PDFs of $\sum_{i=2}^{N} \widebar{\alpha}_i \widebar{Y}_i$ differ across ten observation cases: $N=2,3,4,5,7,10,15,22$. Let $\Delta t=\frac{T}{N-1}, t_1=0$ and $t_i=(i-1)\Delta t$ for $i=2,\ldots,N$. We assume that the parameters of the SDE in~\eqref{sde_xt} are given by the following: $S_0=2, \mu=0.6, \sigma=0.1$ and $\kappa=0.5$, and $\alpha$ is then computed using~\eqref{alpha}. The parameters $\widebar{\mu}_i, \widebar{\sigma}_i, \nu$ and $ \widebar{\delta}_i$ are then computed using~\eqref{mean_z},~\eqref{variance_z},~\eqref{degree_of_freedom} and~\eqref{noncentrality}, respectively. The value of $\widebar{\alpha}_i$ is obtained from~\eqref{alpha_i}, which implies the value of $\widebar{\beta}$. By using these values, the coefficients $\widebar{c}_k$ can be computed via the recurrence relations given in~\eqref{recurrentc0}--\eqref{recurrentdj}. We then obtain the various shapes of the PDFs of $\sum_{i=2}^{N} \widebar{\alpha}_i \widebar{Y}_i$ as shown in the following Figure~\ref{1a}--\ref{1b}.

\begin{figure}[!ht]
    \centering
    \begin{subfigure}[!ht]{0.495\textwidth}
        \includegraphics[width=\textwidth]{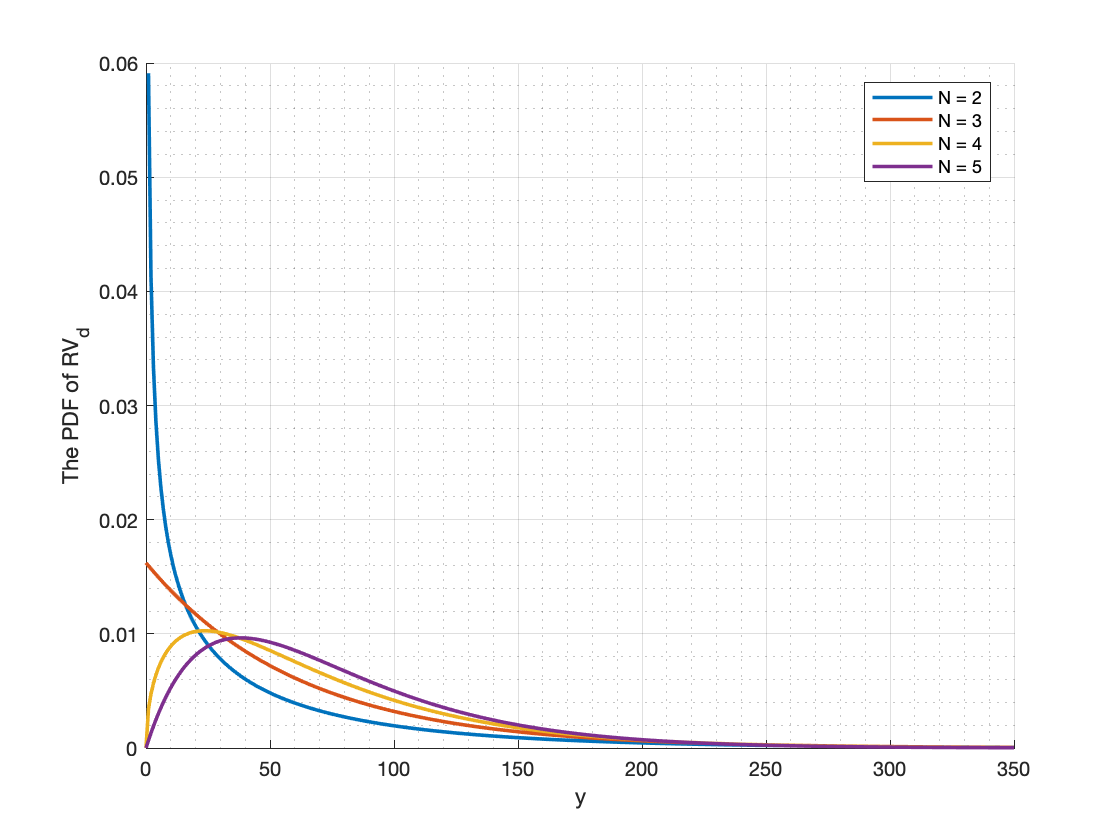}
        \caption{The PDFs of $\sum_{i=2}^{N} \widebar{\alpha}_i \widebar{Y}_i$ for $N=2,\ldots,5$.}
        \label{1a}
    \end{subfigure}
    \begin{subfigure}[!ht]{0.495\textwidth}
        \includegraphics[width=\textwidth]{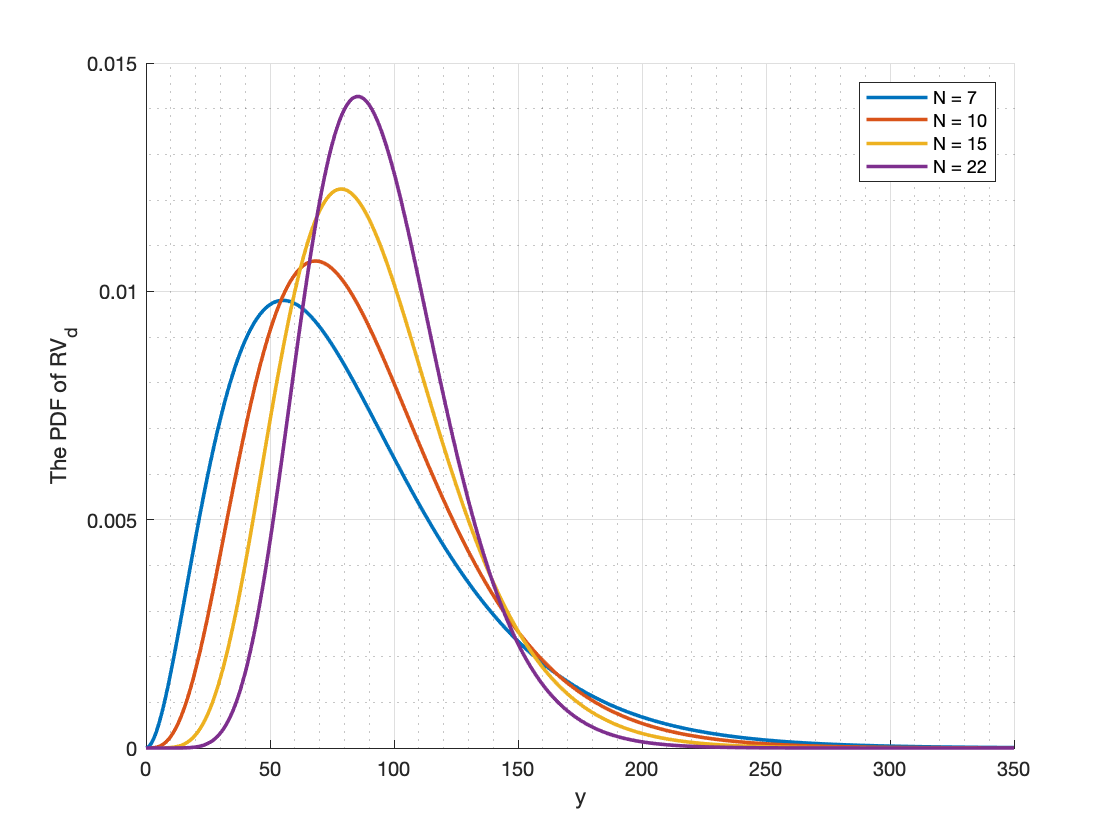}
        \caption{The PDFs of $\sum_{i=2}^{N} \widebar{\alpha}_i \widebar{Y}_i$ for $N=7,10,15,22$.}
        \label{1b}
    \end{subfigure}
    \caption{The PDFs of $\sum_{i=2}^{N} \widebar{\alpha}_i \widebar{Y}_i$ in which $\widebar{Y}_i \sim \mathcal{NC}_{\chi^2_1}(\widebar{\delta}_i)$ for $i=2,\ldots,N$, and the value of parameters $\widebar{\alpha}_i$ and $\widebar{\delta}_i$ can be computed by~\eqref{alpha_i} and~\eqref{noncentrality}, respectively.}
\label{fig1}
\end{figure}

\begin{Example}
    To demonstrate the accuracy and efficiency of the Laguerre expansions for the PDFs of $RV_d$ as defined in~\eqref{pdf_rv}, we compare the analytical expression $\widebar{f}_\nu^{(\widebar{\beta},\widebar{\mu}_0)} (y)$ with the approximate PDFs of $RV_d$ obtained from MC simulations. In this example, the approximate PDFs are constructed from $N_p$ sample paths generated by simulating the SDE~\eqref{sde_xt}. To ensure high accuracy in the MC-based approximation, we set $N_p=10^5$.
\end{Example}

To compare the analytical and approximate PDFs of $RV_d$, we set all parameters to the same values as in Example~\ref{ex.1}, except for the price volatility $\sigma$, which is set to $0.08$, and $0.1$, and the number of observations $N$, which is set to $52$, and $252$. The results are presented in the following figure.
\begin{figure}[!ht]
    \centering
    \begin{subfigure}[!ht]{0.496\textwidth}
        \includegraphics[width=\textwidth]{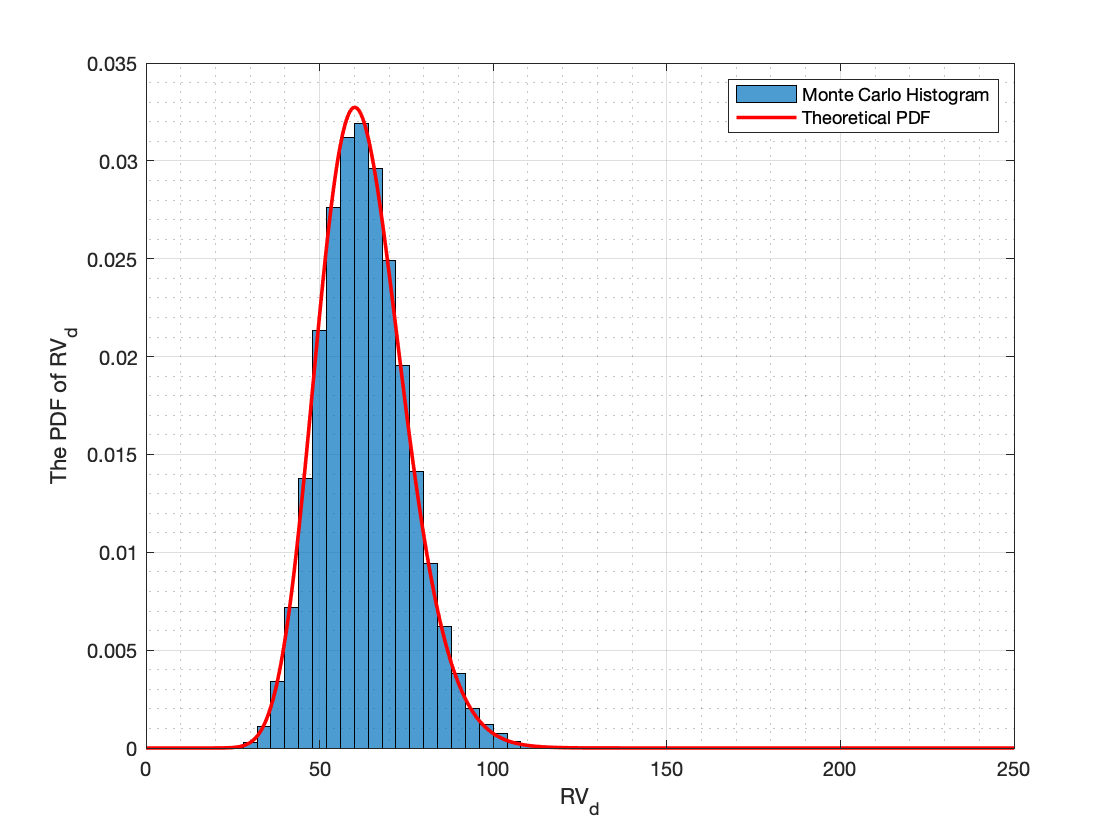}
        \caption{$RV_d: N=52, \sigma=0.08$.}
        \label{2a}
    \end{subfigure}
        \begin{subfigure}[!ht]{0.496\textwidth}
        \includegraphics[width=\textwidth]{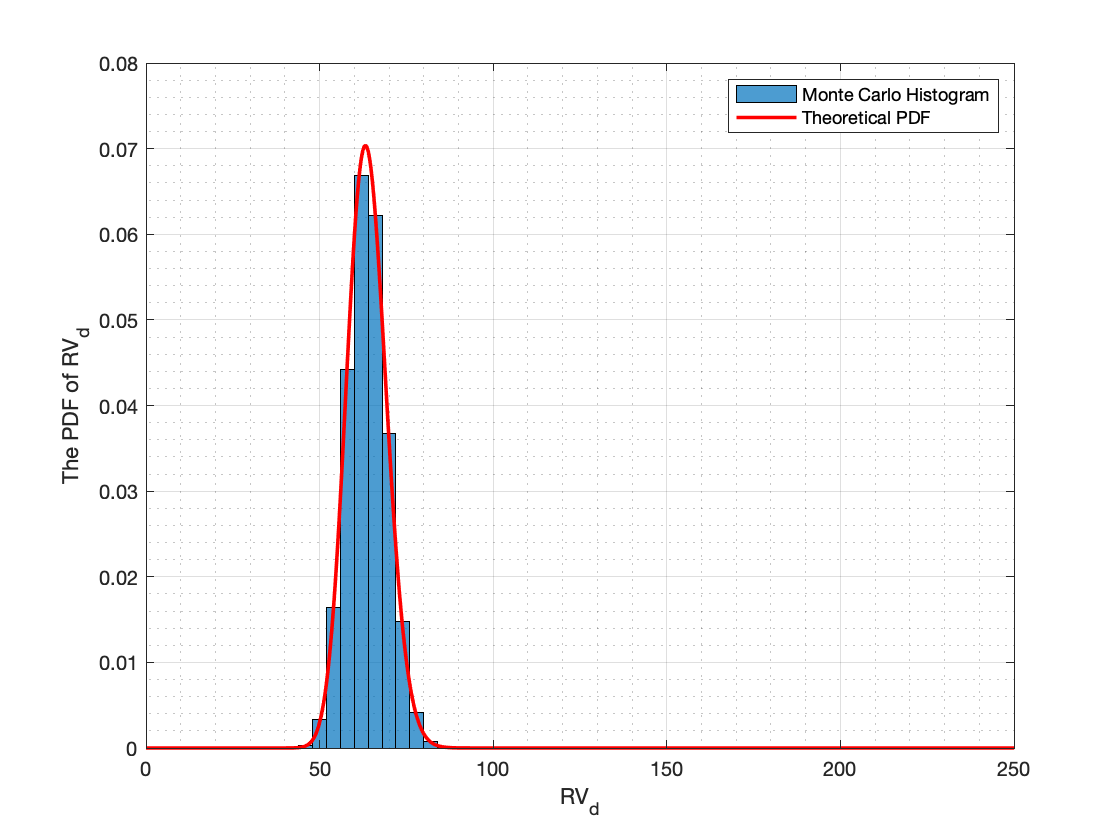}
        \caption{$RV_d: N=252, \sigma=0.08$.}
        \label{2b}
    \end{subfigure}
        \begin{subfigure}[!ht]{0.496\textwidth}
        \includegraphics[width=\textwidth]{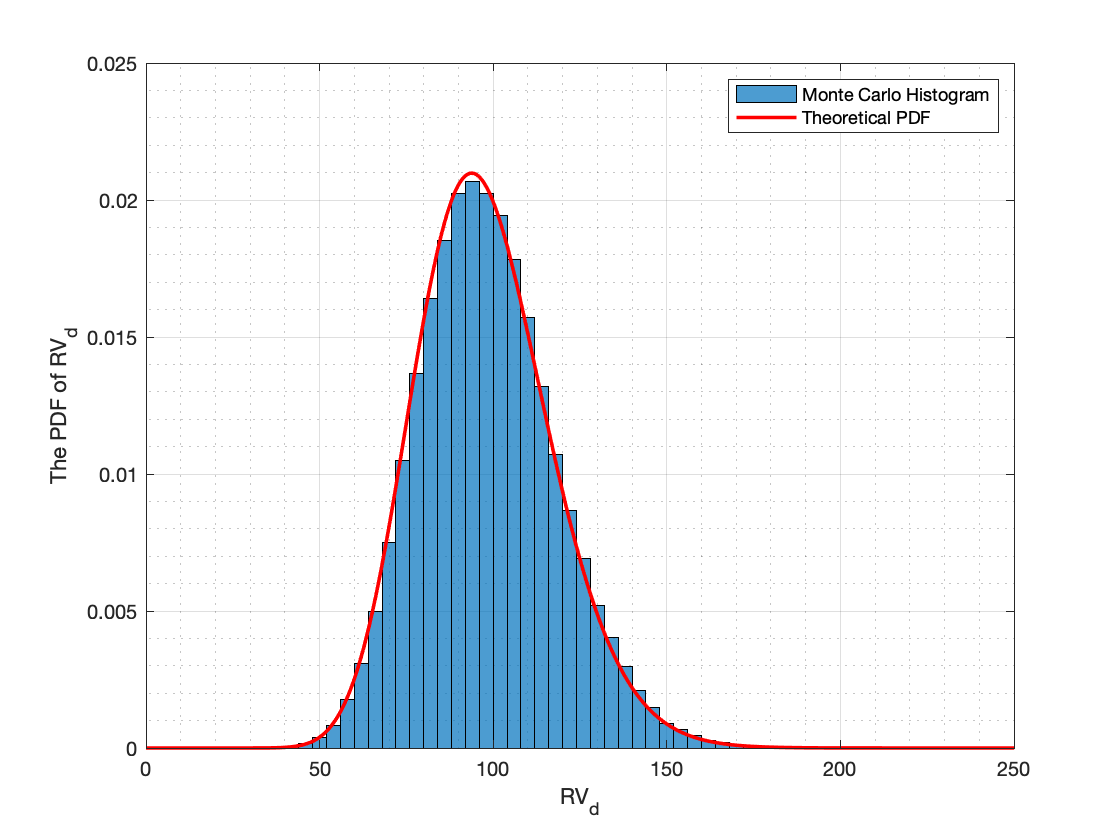}
        \caption{$RV_d: N=52, \sigma=0.1$.}
        \label{2c}
    \end{subfigure}
        \begin{subfigure}[!ht]{0.496\textwidth}
        \includegraphics[width=\textwidth]{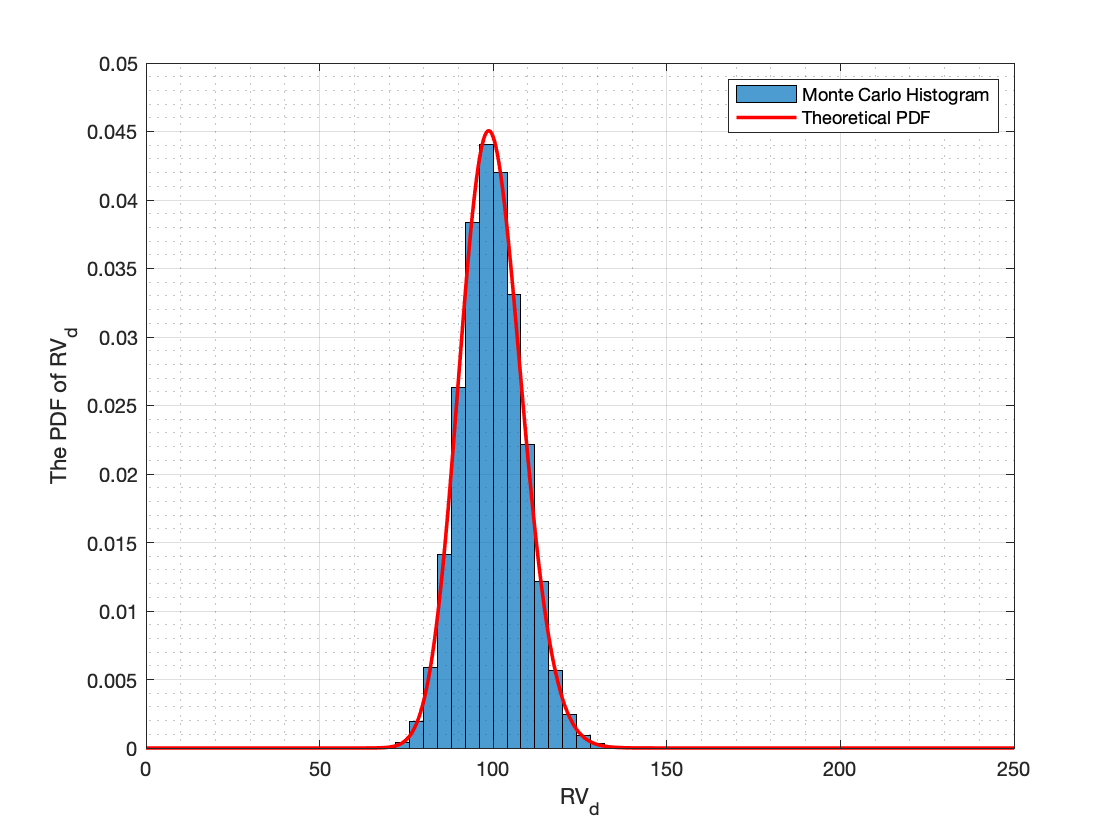}
        \caption{$RV_d: N=252, \sigma=0.1$.}
        \label{2d}
    \end{subfigure}
    \caption{The PDF of $RV_d(t_1,N,T)$ obtained by using the Laguerre expansion~\eqref{pdf_rv} whereas the corresponding histogram computed by MC simulation.}
\label{fig2}
\end{figure}
It is clearly observed that the PDFs of $RV_d$ obtained from the Laguerre expansions~\eqref{pdf_rv} are consistent with those corresponding histograms obtained from MC simulations for all $N=52,252$ and $\sigma=0.08,0.10$, as shown in Figure~\ref{2a}--\ref{2d}.

\begin{Example} \label{ex.3}
    In this example, we investigate the absolute values of truncation errors on fair strike prices of $\sqrt{RV_d}$ as presented in Theorem~\ref{thm_abs_error}. For given $K+1$ as a number of terms in the finite summation series~\eqref{KvolSwap1_1}.
\end{Example}

By setting all parameters to the same values as in Example~\ref{ex.1}, except for $N = 252$ and varying $\sigma$ which are set to $0.05$, $0.06$, $0.07$, $0.08$, $0.09$, and $0.1$, we evaluate the accuracy and efficiency of $K^1_{vol_1}$. Specifically, five sequences of truncation errors $\varepsilon_{K,\infty}^{(\widebar{\beta}, \widebar{\mu}_0)} (\ell,\nu)$ are computed for three different values of the degree of mean reversion $\kappa$, i.e. $\kappa_1=0.5, \kappa_2=1.5$, and $\kappa_3=3.0$, using equation~\eqref{abs_error}, as summarized in Table~\ref{table_truncation_error}. By choosing $K = 3$, we obtain
\begin{equation*}
    \left| \varepsilon_{3,\infty}^{(\widebar{\beta}, \widebar{\mu}_0)} (\ell,\nu) \right| = \left| K^1_{vol_1}-K^1_{{vol_1},3} \right| < 2.3101\text{E}-08
\end{equation*}
for all values of $\sigma$ and $\kappa$, indicating that $K = 3$ yields sufficiently accurate results. Therefore, we adopt $K = 3$ in our subsequent experiments.

\begin{table}[h]
    \centering
    \begin{tabular}{cccccccc}
        \toprule
        $\kappa$ & $K$ & $\sigma = 0.05$ & $\sigma = 0.06$ & $\sigma = 0.07$ & $\sigma = 0.08$ & $\sigma = 0.09$ & $\sigma = 0.10$ \\
        \midrule
        \multirow{4}{*}{0.5}
        & 0 & 5.0621E-03 & 4.0908E-03 & 3.3784E-03 & 2.8286E-03 & 2.3877E-03 & 2.0238E-03 \\
        & 1 & 2.4812E-06 & 1.3472E-06 & 7.8533E-07 & 4.7996E-07 & 3.0264E-07 & 1.9458E-07 \\
        & 2 & 2.3443E-09 & 8.5308E-10 & 3.4989E-10 & 1.5546E-10 & 7.2847E-11 & 3.5296E-11 \\
        & 3 & 2.6654E-12 & 6.4748E-13 & 1.8563E-13 & 6.0396E-14 & 1.9540E-14 & 7.1054E-15 \\
        \midrule
        \multirow{4}{*}{1.5}
        & 0 & 2.2769E-02 & 1.8416E-02 & 1.5221E-02 & 1.2749E-02 & 1.0760E-02 & 9.1092E-03 \\
        & 1 & 5.0288E-05 & 2.7367E-05 & 1.5987E-05 & 9.7862E-06 & 6.1741E-06 & 3.9644E-06 \\
        & 2 & 2.1434E-07 & 7.8326E-08 & 3.2261E-08 & 1.4386E-08 & 6.7571E-09 & 3.2735E-09 \\
        & 3 & 1.1008E-09 & 2.6957E-10 & 7.8065E-11 & 2.5270E-11 & 8.7965E-12 & 3.1957E-12 \\
        \midrule
        \multirow{4}{*}{3.0}
        & 0 & 4.8567E-02 & 3.9763E-02 & 3.3365E-02 & 2.8470E-02 & 2.4577E-02 & 2.1387E-02 \\
        & 1 & 2.2947E-04 & 1.2806E-04 & 7.7190E-05 & 4.9107E-05 & 3.2476E-05 & 2.2088E-05 \\
        & 2 & 2.0955E-06 & 7.9627E-07 & 3.4435E-07 & 1.6308E-07 & 8.2467E-08 & 4.3744E-08 \\
        & 3 & 2.3101E-08 & 5.9702E-09 & 1.8497E-09 & 6.5101E-10 & 2.5122E-10 & 1.0367E-10 \\
        \bottomrule
    \end{tabular}
    \caption{The values of $\left| \varepsilon_{K,\infty}^{(\widebar{\beta}, \widebar{\mu}_0)} (\ell,\nu) \right| $, evaluated by~\eqref{error} with $\ell=1/2$ and $N=252$, are computed based on the parameter settings in Example~\ref{ex.3}, with different values of price volatility $\sigma$ and degree of the mean reversions $\kappa$.}
    \label{table_truncation_error}
\end{table}

By setting $\sigma = 0.5$, we plot the convergence of MC volatility swap prices to $K^1_{{vol_1},3}$ for $\kappa_i$, $i=1,2,3$. Three sequences of MC volatility swap prices, with $N_p$ varying from $1 \times 10^3$ to $1 \times 10^5$, are plotted against $K^1_{{vol_1},3}$ for each value of $\kappa$, as shown in Figures~\ref{3a},~\ref{3c},~and~\ref{3e}, respectively. Additionally, we plot the corresponding sequences of variances of the MC volatility swap prices for $\kappa_i$, $i=1,2,3$, in Figures~\ref{3b},~\ref{3d},~and~\ref{3f}, respectively.
\begin{figure}[!ht]
    \centering
    \begin{subfigure}[!ht]{0.495\textwidth}
        \includegraphics[width=\textwidth]{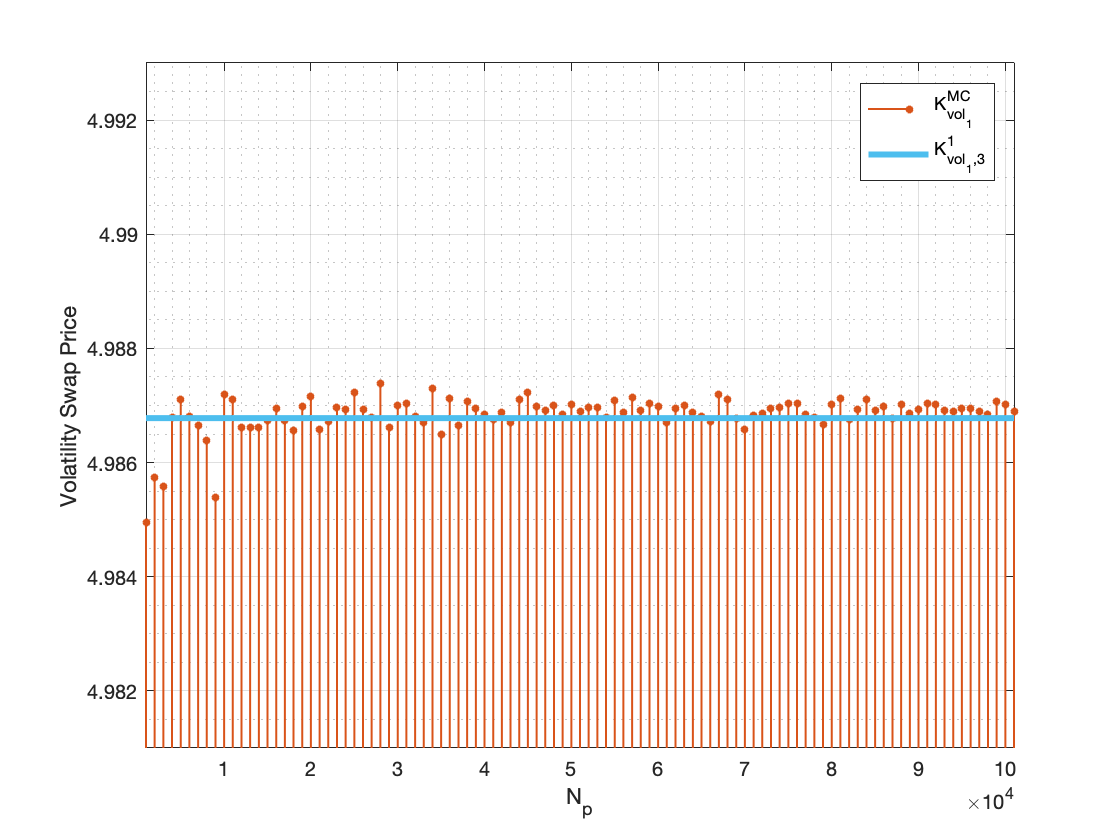}
        \caption{Volatility swap prices when $\kappa=0.5$}
        \label{3a}
    \end{subfigure}
    \begin{subfigure}[!ht]{0.495\textwidth}
        \includegraphics[width=\textwidth]{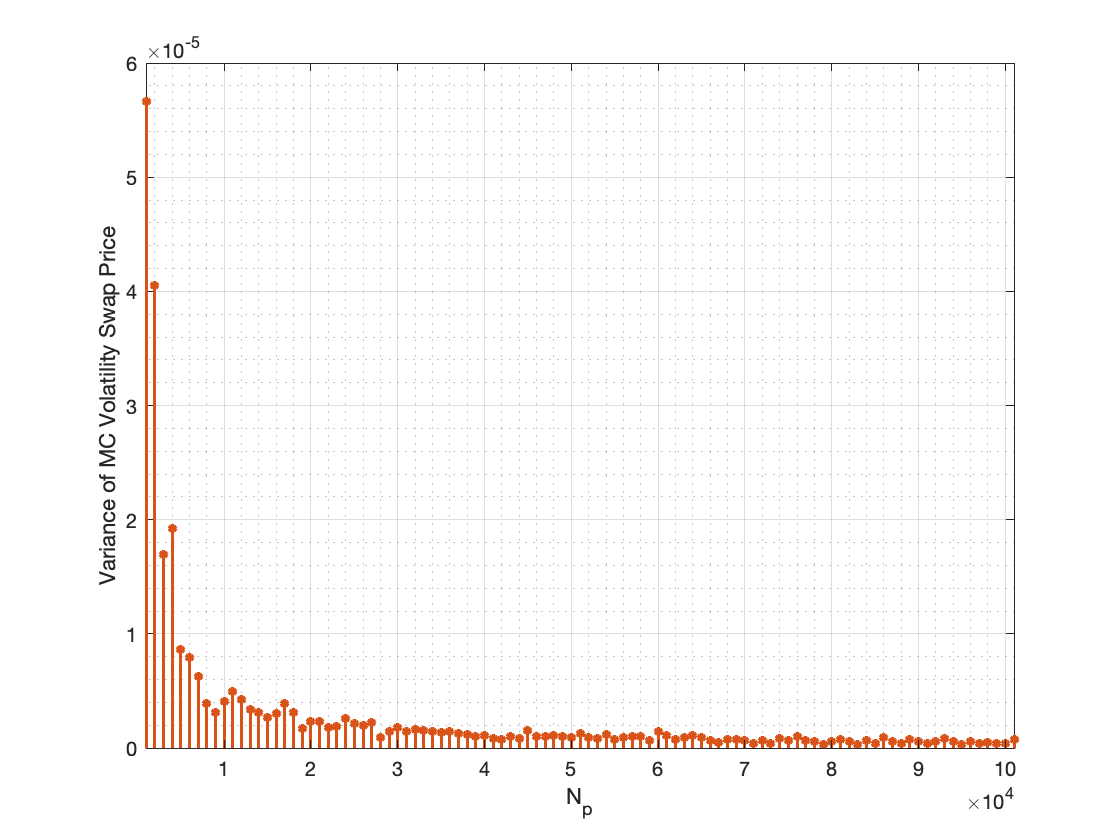}
        \caption{Variances of volatility swap price when $\kappa=0.5$}
        \label{3b}
    \end{subfigure}
    \begin{subfigure}[!ht]{0.495\textwidth}
        \includegraphics[width=\textwidth]{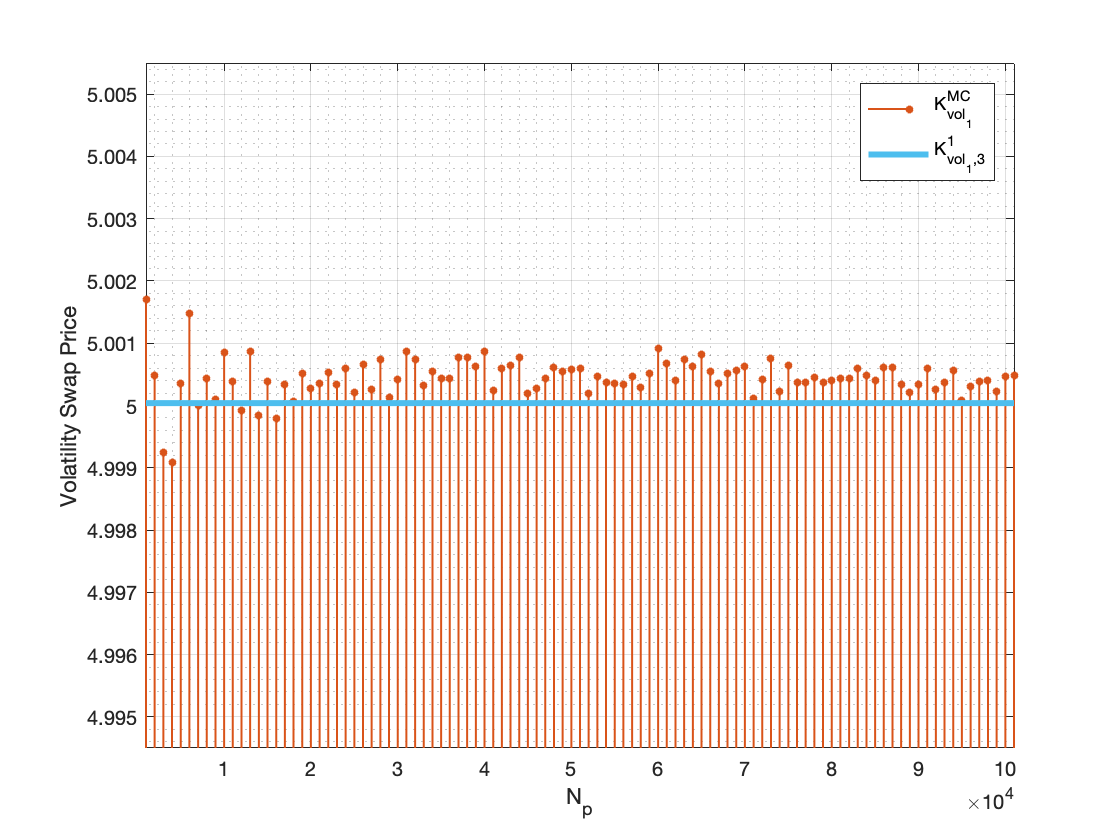}
        \caption{Variances of volatility swap price when $\kappa=1.5$}
        \label{3c}
    \end{subfigure}
    \begin{subfigure}[!ht]{0.495\textwidth}
        \includegraphics[width=\textwidth]{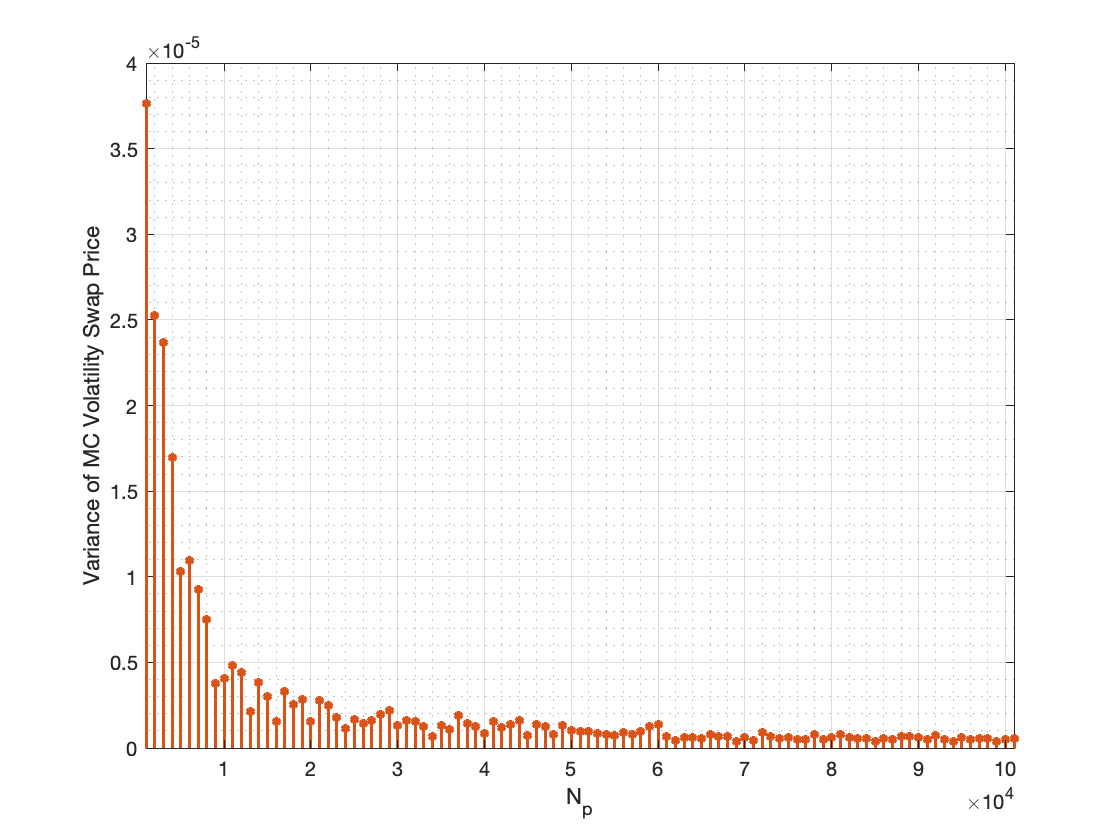}
        \caption{Variances of volatility swap price when $\kappa=1.5$}
        \label{3d}
    \end{subfigure}
    \begin{subfigure}[!ht]{0.495\textwidth}
        \includegraphics[width=\textwidth]{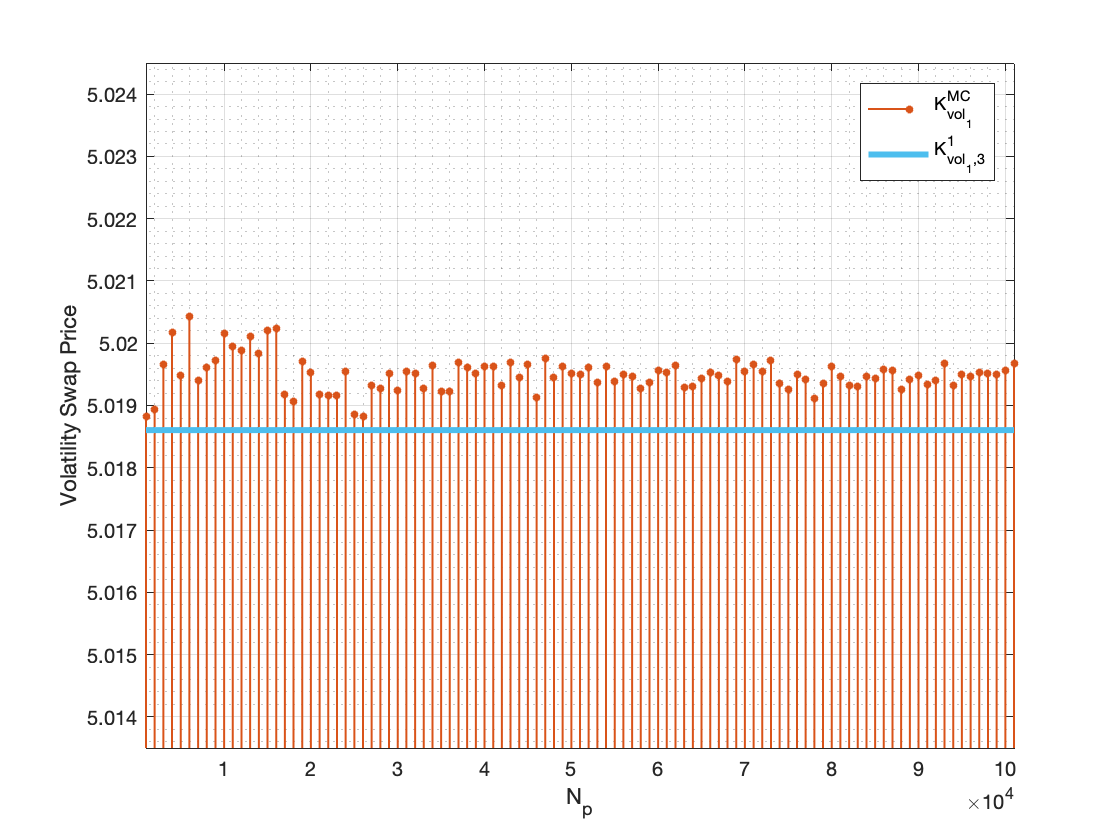}
        \caption{Variances of volatility swap price when $\kappa=3.0$}
        \label{3e}
    \end{subfigure}
    \begin{subfigure}[!ht]{0.495\textwidth}
        \includegraphics[width=\textwidth]{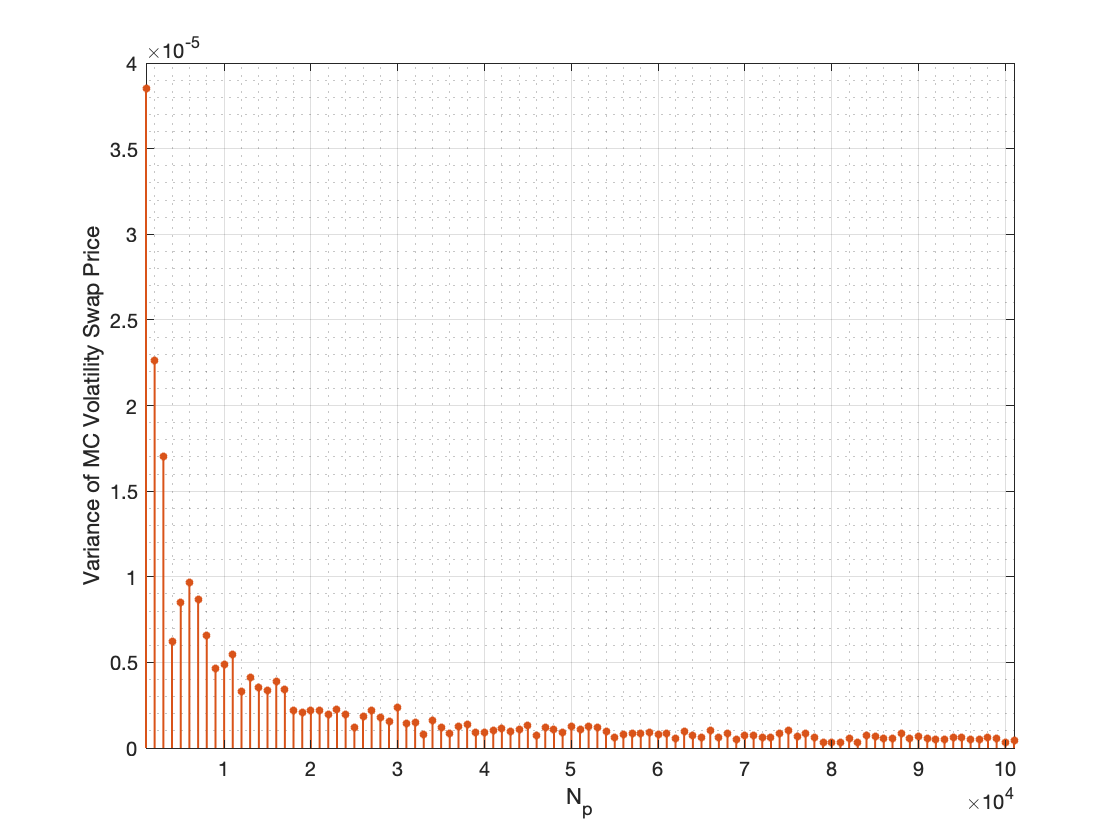}
        \caption{Variances of volatility swap price when $\kappa=3.0$}
        \label{3f}
    \end{subfigure}
    \caption{The convergences of the MC volatility swap prices compared with $K^1_{{vol_1},3}$ computed using~\eqref{KvolSwap1_1}, and the convergences the variances of the MC volatility swap prices to zero.}
\label{fig3}
\end{figure}

As shown in Figures~\ref{3a}--\ref{3b}, the MC volatility swap prices converge to $K^1_{{vol_1},3}$, and the variances of the MC volatility swap prices approach zero, confirming that the MC estimates are very close to the true value of $K^1_{{vol_1},3}$. Similar convergence behavior is observed for $\kappa_2 = 1.5$ and $\kappa_3 = 3$, as illustrated in Figures~\ref{3c}--\ref{3d} and Figures~\ref{3e}--\ref{3f}, respectively. Moreover, as shown in Figures~\ref{3a},~\ref{3c},~and~\ref{3e}, it can be observed that increasing the value of $\kappa$ results in slower convergence of the MC simulation results toward the analytical values.


\subsection{Effects on fair strike prices of volatility and variance swaps}
\medskip

\begin{Example} \label{ex.4}
    This example is studied on the effects of price volatility $\sigma$ in Schwartz one-factor model~\eqref{sde_xt} on fair strike prices of volatility and variance swaps presented in Theorems~\ref{thm_Kvol1_1} and~\ref{thm_Kvar1_1}, respectively.
\end{Example}

Let $T=1,\Delta t=\frac{T}{N-1}, t_1=0$ and $t_i=(i-1)\Delta t$ for $i=2,\ldots,N$. To measure the effects of the price volatility $\sigma$ on fair strike prices of volatility and variance swaps changing in the closed interval of time $[t_1,T]$, we set $S_0=2$ and $\mu=1$. We then obtain a variation of $Z_{i} = X_{t_{i}}-X_{t_{i-1}} = \ln{S_{t_{i}}}-\ln{S_{t_{i-1}}}$ under $\kappa_i$, for $i=1,2,3$ while $\sigma$ is varied over the interval $[0.005,0.1]$. With this parameter setting, we let $N=52$ and then we plot the volatility swap $K^1_{vol_1}$ obtained from~\eqref{KvolSwap1_1} against those obtained from MC simulations with $N_p=10^5$ sample paths denoted by $K^{MC}_{vol_1}$. Similarly, the variance swap $K^1_{var_1}$, obtained from~\eqref{KvarSwap1_1}, is compared with the MC simulation results $K^{MC}_{var_1}$ using the same number of paths. 

\begin{figure}[!ht]
    \centering
    \begin{subfigure}[!ht]{0.495\textwidth}
        \includegraphics[width=\textwidth]{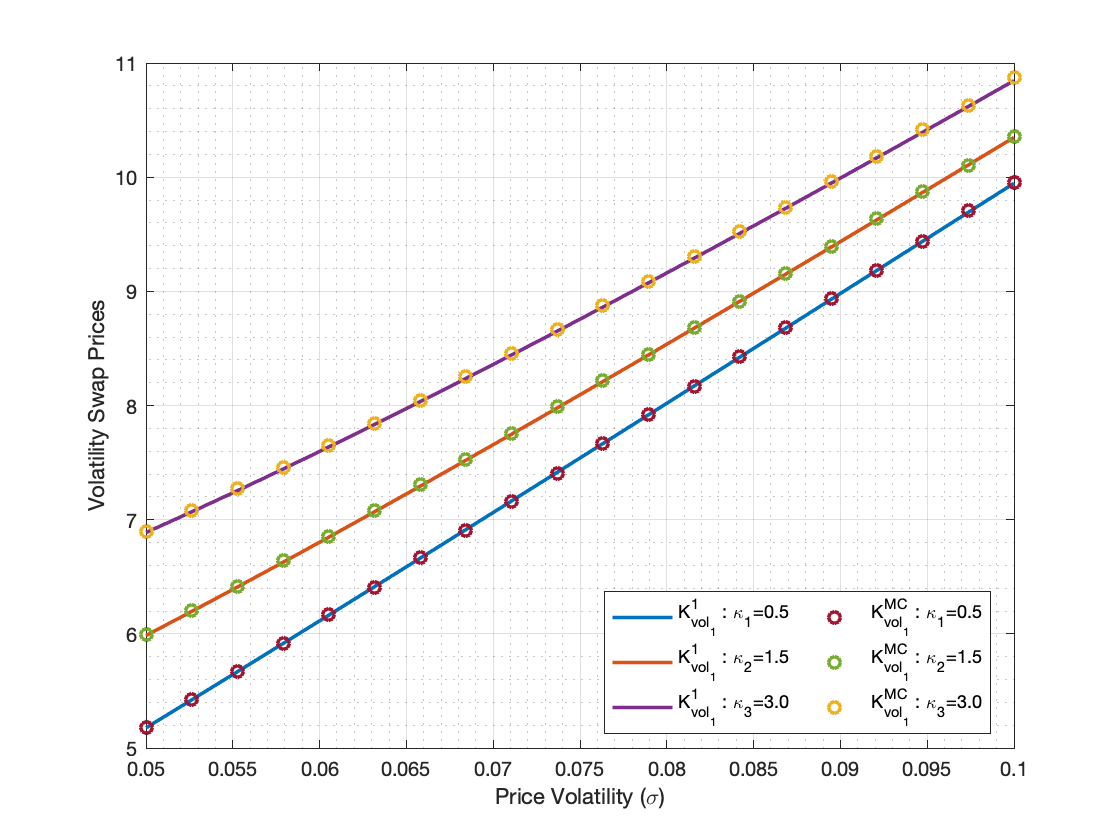}
        \caption{The volatility swap prices computed by using~\eqref{KvolSwap1_1} with $\kappa_i,$ $i=1,2,3$ and MC simulations.}
        \label{4a}
    \end{subfigure}
    \begin{subfigure}[!ht]{0.495\textwidth}
        \includegraphics[width=\textwidth]{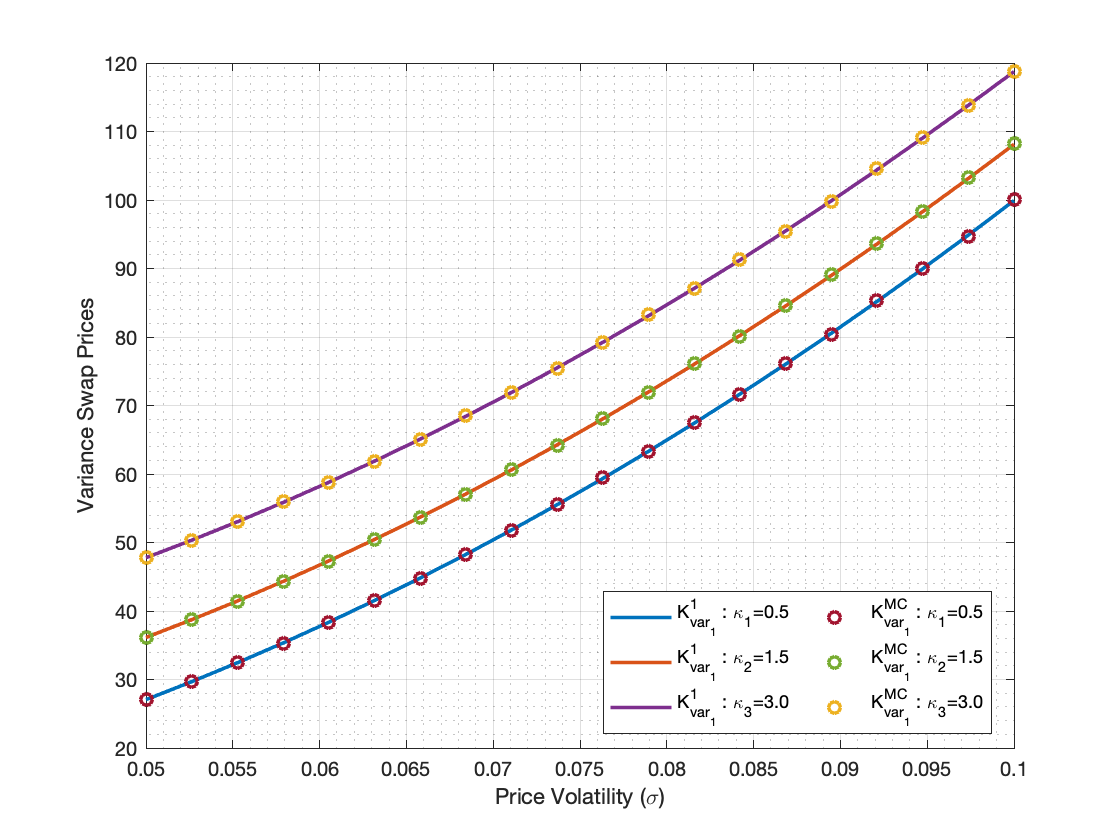}
        \caption{The variance swap prices computed by using~\eqref{KvarSwap1_1} with $\kappa_i,$ $i=1,2,3$ and MC simulations.}
        \label{4b}
    \end{subfigure}
    \caption{The comparison between the variance and variance swap prices computed by using our closed-form formula~\eqref{KvarSwap1_1} and~\eqref{KvolSwap1_1}, respectively, with the difference of prices volatility ($\sigma$) and degree of means reversion ($\kappa$), and the MC simulations.}
\label{fig4}
\end{figure}

The numerical results clearly demonstrate that the closed-form pricing formulas~\eqref{KvarSwap1_1} and~\eqref{KvolSwap1_1} yield values that perfectly match the results from MC simulations. Moreover, Figure~\ref{4a} shows that an increase in price volatility results in the fair strike prices of volatility swaps increasing for each $\kappa_i$ all $i=1,2,3$. For the fair strike prices of variance swaps, similar results are shown in Figure~\ref{4b}.

\begin{Example}
    In this example, we aim to examine the effects of the number of trading days $N$ on fair strike prices of volatility and variance swaps, as presented in Theorems~\ref{thm_Kvol1_1} and~\ref{thm_Kvar1_1}, respectively.  
\end{Example}

To study the effects of the number of trading days $N$ on fair strike prices of volatility and variance swaps, all parameters are set to be the same values as used in Example~\ref{ex.4}, except that we fix $\kappa=\kappa_1$ and consider three values of price volatility $\sigma:\sigma_1=0.05,\sigma_2=0.06$, and $\sigma_3=0.07$ . The number of observations $N$ is varied from $2$ to $252$. We compare the results by plotting $K^1_{vol_1}$ against $K^{MC}_{vol_1}$ and $K^1_{var_1}$ against $K^{MC}_{var_1}$ as shown in the following figure. 

\begin{figure}[!ht]
    \centering
    \begin{subfigure}[!ht]{0.495\textwidth}
        \includegraphics[width=\textwidth]{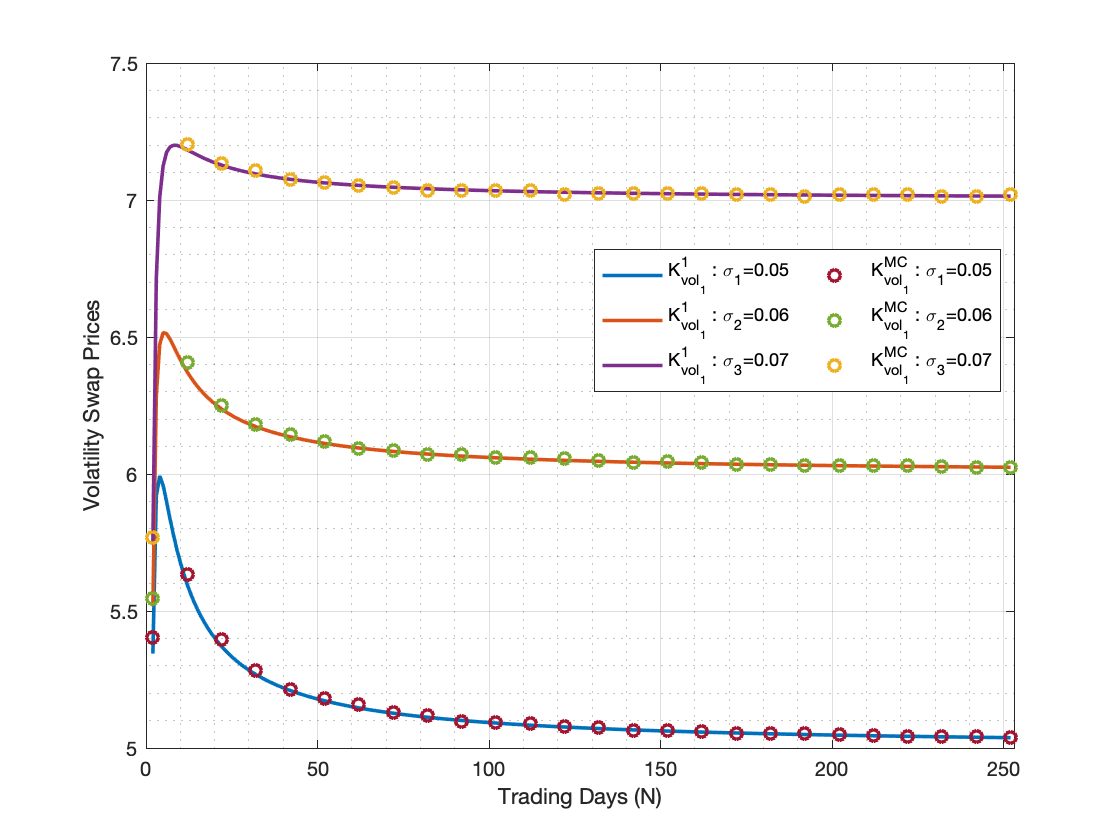}
        \caption{The volatility swap prices computed by using~\eqref{KvolSwap1_1} with $\sigma_i,$ $i=1,2,3,$ and the MC simulations.}
        \label{5a}
    \end{subfigure}
    \begin{subfigure}[!ht]{0.495\textwidth}
        \includegraphics[width=\textwidth]{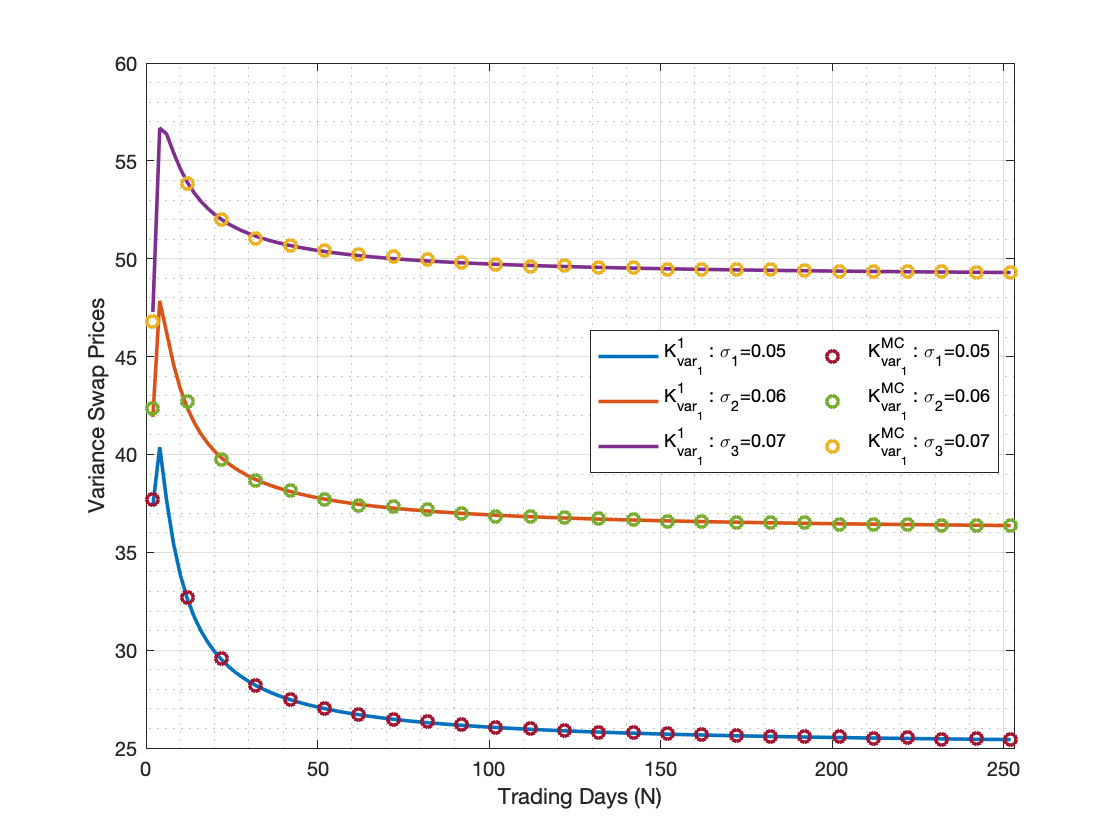}
        \caption{The variance swap prices computed by using~\eqref{KvarSwap1_1} with $\sigma_i,$ $i=1,2,3,$ and the MC simulations.}
        \label{5b}
    \end{subfigure}
    \caption{The comparison between the variance and variance swap prices computed by using our closed-form formulas~\eqref{KvarSwap1_1} and~\eqref{KvolSwap1_1}, respectively, with the difference number of trading days ($N$) and prices volatility ($\sigma$), and the MC simulations.}
\label{fig5}
\end{figure}

Figure~\ref{5a} illustrates that as the number of trading days increases, the fair strike prices of the volatility swap obtained from our closed-form pricing formula~\eqref{KvolSwap1_1} are consistent and converge to the results obtained from MC simulations for all $\sigma_i$ for $i = 1, 2, 3$. A similar convergence pattern is observed for the variance swap prices, as shown in Figure~\ref{5b}, where the prices are computed by~\eqref{KvarSwap1_1}.


\section{Conclusion} \label{sec_conclusion}
\medskip

In this paper, we first derived the distribution of the realized variance of log-returns with time-varying volatility for pricing volatility swaps under discrete-time observations, assuming the underlying asset follows the Schwartz one-factor model. We demonstrated that the realized variance is in the class of generalized noncentral chi-square distributions, as introduced by Koutras~\cite{koutras1986generalized} in 1986, and can be expressed as a linear combination of independent noncentral chi-square random variables with weighted parameters. Based on this result, we derived the first analytical pricing formulas for volatility and variance swaps, as well as for volatility and variance options, under both time-varying and constant volatility cases. To support practical implementation, we provided an error analysis and proposed simple closed-form approximations for pricing volatility derivatives under the Schwartz model. Numerical experiments were conducted, clearly demonstrating the accuracy and computational efficiency of the proposed formulas and their consistency with Monte Carlo simulations. Finally, we analyzed the effects of price volatility and the number of trading days on the fair strike prices of volatility and variance swaps. The results show that increasing price volatility leads to higher fair strike prices, while increasing the number of trading days causes the fair strike prices to converge to their true values. Overall, the newly derived analytical formulas significantly enhance computational efficiency, particularly for pricing volatility derivatives with square-root payoffs. This contributes a practical and effective tool for both researchers and practitioners in the financial derivatives market.


\section*{Acknowledgement}
We would like to express our sincere gratitude to the Development and Promotion of Science and Technology Talents Project (DPST) for supporting this research. We also thank the anonymous reviewers for their valuable comments and recommendations, which have helped to improve the quality and clarity of this work. Any remaining errors are solely our own.










\end{document}